\documentclass{lms}
\usepackage{amssymb}
\usepackage{enumerate}
\usepackage{amsmath}
\usepackage[backref=page,colorlinks=true]{hyperref} 
\usepackage{tikz}
\usetikzlibrary{patterns}
\usepackage[font=footnotesize,margin={8pt,8pt}]{caption}

\makeatletter

\def\myMRbibitem{\@ifnextchar[\my@lbibitem\my@bibitem}

\def\mybiblabel#1#2{\@biblabel{{\hyperref{http://www.ams.org/mathscinet-getitem?mr=#1}{}{}{#2}}}}

\def\myhyperanchor#1{\Hy@raisedlink{\hyper@anchorstart{cite.#1}\hyper@anchorend}}

\def\my@lbibitem[#1]#2#3#4\par{%
    \item[\mybiblabel{#2}{#1}\myhyperanchor{#3}\hfill]#4%
    \@ifundefined{ifbackrefparscan}{}{\BR@backref{#3}}%
    \if@filesw{\let\protect\noexpand\immediate
       \write\@auxout{\string\bibcite{#3}{#1}}}\fi\ignorespaces%
}

\def\my@bibitem#1#2#3\par{%
    \refstepcounter\@listctr
    \item[\mybiblabel{#1}{\the\value\@listctr}\myhyperanchor{#2}\hfill]#3%
    \@ifundefined{ifbackrefparscan}{}{\BR@backref{#2}}%
    \if@filesw\immediate\write\@auxout
        {\string\bibcite{#2}{\the\value\@listctr}}\fi\ignorespaces%
}

\makeatother

\renewcommand*{\backref}[1]{}
\renewcommand*{\backrefalt}[4]{\quad \tiny 
    \ifcase #1 (Not cited.)%
    \or        (Cited on page~#2.)%
    \else      (Cited on pages~#2.)%
    \fi}

\newtheorem{theorem}{Theorem} 
\newtheorem{otherthm}{Theorem}[section] 
\newtheorem{lemma}[otherthm]{Lemma}     
\newtheorem{corollary}[otherthm]{Corollary}
\newtheorem{proposition}[otherthm]{Proposition}

\newnumbered{assertion}{Assertion}[section]    
\newnumbered{conjecture}{Conjecture}[section]  
\newnumbered{definition}{Definition}[section]
\newnumbered{remark}{Remark}[section]
\newnumbered{note}{Note}[section]
\newnumbered{question}{Question}[section]
\newnumbered{example}{Example}[section]
\newunnumbered{claim}{Claim}
\newunnumbered{ack}{Acknowledgements}

\providecommand{\bysame}{\makebox[3em]{\hrulefill}\thinspace} 

\newcommand{\R}{\mathbb{R}}
\newcommand{\Z}{\mathbb{Z}}

\newcommand{\N}{\mathbb{N}}

\newcommand{\sA}{\mathsf{A}}
\newcommand{\sB}{\mathsf{B}}
\newcommand{\sK}{\mathsf{K}}
\newcommand{\cD}{\mathcal{D}}
\newcommand{\cC}{\mathcal{C}}
\newcommand{\cK}{\mathcal{K}}

\newcommand{\cU}{\mathcal{U}}
\newcommand{\cR}{\mathcal{R}}

\DeclareMathOperator{\Int}{Int}

\newcommand{\tr}{\mathrm{tr}}

\newcommand{\subrad}{\check{\varrho}}
\newcommand{\jrad}{\hat{\varrho}}
\DeclareMathOperator{\codim}{codim}

\newcommand{\GL}{\mathit{GL}}
\newcommand{\SL}{\mathit{SL}}

\newcommand{\wed}{{\mathord{\wedge}}} 

\newcommand{\arxiv}[1]{Preprint \href{http://arxiv.org/abs/#1}{arXiv:{#1}}}

\numberwithin{equation}{section}         
\renewcommand{\figurename}{Fig{.}}

\hypersetup{bookmarksdepth = 2} 

\renewcommand{\epsilon}{\varepsilon}
\renewcommand{\setminus}{\smallsetminus}
\renewcommand{\angle}{\measuredangle}

\begin{document}

\title[Continuity of the lower spectral radius]{Continuity properties of the lower spectral radius}

\author{Jairo Bochi and Ian D.~Morris}

\date{\today}

\classno{15A60 (primary), 15B48, 37H15, 37D30, 47D03, 65F99 (secondary)}

\extraline{The first named author was partially supported by CNPq and Faperj during the preparation of this paper. The second named author gratefully acknowledges the hospitality of PUC-Rio and the support of the University of Surrey FEPS Research Support Fund.}

\maketitle

\abstract{
The lower spectral radius, or joint spectral subradius, of a set of real $d \times d$ matrices is defined to be the smallest possible exponential growth rate of long products of matrices drawn from that set. The lower spectral radius arises naturally in connection with a number of topics including combinatorics on words, the stability of linear inclusions in control theory, and the study of random Cantor sets. In this article we apply some ideas originating in the study of dominated splittings of linear cocycles over a dynamical system to  characterise the points of continuity of the lower spectral radius on the set of all compact sets of invertible $d \times d$ matrices. As an application we exhibit open sets of pairs of $2 \times 2$ matrices within which the analogue of the Lagarias--Wang finiteness property for the lower spectral radius fails on a residual set, and discuss some implications of this result for the computation of the lower spectral radius.
}

\section{Introduction}

\subsection{Background}
Recall that the spectral radius of a $d \times d$ real matrix $A$, which we shall denote by $\rho(A)$, is defined to be the maximum of the moduli of the eigenvalues of $A$ and satisfies Gelfand's formula
\[\rho(A)=\lim_{n \to \infty} \left\|A^n\right\|^{\frac{1}{n}}=\inf_{n \geq 1} \left\|A^n\right\|^{\frac{1}{n}}\]
for every operator norm $\|\mathord{\cdot}\|$ on the set $M_d(\R)$ of all $d \times d$ matrices. By analogy with this formula, the \emph{joint spectral radius} or \emph{upper spectral radius} of a bounded nonempty set $\sA$ of $d \times d$ matrices was defined by G.-C. Rota and G. Strang (\cite{RS}, reprinted in \cite{R03}) to be the quantity
$$
\jrad(\sA):=\lim_{n \to \infty} \sup\left\{\left\|A_n\cdots A_1\right\|^{\frac{1}{n}}\colon A \in \sA\right\}=\inf_{n \geq 1} \sup\left\{\left\|A_n\cdots A_1\right\|^{\frac{1}{n}}\colon A \in \sA\right\}
$$
which is likewise independent of the choice of operator norm $\|\mathord{\cdot}\|$ on $M_d(\R)$. Interest in the upper spectral radius was subsequently stimulated by applications in diverse areas such as control theory \cite{Ba,Gur95}, wavelet regularity \cite{DL,S}, combinatorics \cite{BCJ,DST} and coding theory \cite{BJP,MOS}. The upper spectral radius is currently the subject of lively research attention, of which we note for example \cite{BM,BN,BTV,DHX13,GP,GWZ,GZ,HMST,Jbook,M13,MS,PJB}
and references therein.

The \emph{joint spectral subradius}, or \emph{lower spectral radius}, of a nonempty set $\sA\subset M_d(\R)$ is similarly defined to be the quantity
$$\subrad(\sA):=\lim_{n \to \infty} \inf\left\{\left\|A_n\cdots A_1\right\|^{\frac{1}{n}}\colon A \in \sA\right\}
=\inf_{n \geq 1} \inf\left\{\left\|A_n\cdots A_1\right\|^{\frac{1}{n}}\colon A \in \sA\right\},
$$
but this concept appears not to have been introduced until much later \cite{Gur95}. Like the upper spectral radius, the lower spectral radius arises naturally in contexts such as control theory \cite{Gur95}, the regularity of fractal structures \cite{DK,Pr04}, and combinatorics \cite{BCJ,Pr00}. Whilst it has also been the subject of recent research attention \cite{BM,GP,J12,PJB} the volume of results is significantly smaller. 

The smaller size of this body of research on the lower spectral radius can perhaps be explained by the fact that the behaviour of the lower spectral radius is significantly less tractable than that of the upper spectral radius. For example, L.~Gurvits has demonstrated in \cite{Gur96} that for a general finite set $\sA$ of real $d \times d$ matrices it is possible to determine whether or not $\jrad(\sA)=0$ using a number of arithmetic operations which depends polynomially on $d$ and on the cardinality of $\sA$. On the other hand, T.~Neary has shown in \cite{N} that the problem of determining whether or not $\subrad(\sA)=0$ is algorithmically undecidable even when $\sA$ consists of a pair of $15 \times 15$ integer matrices (for related earlier results see also \cite{Pat,TB1,TB2}). Similarly, the upper spectral radius was shown in 1995 by C.~Heil and G.~Strang \cite{HS} to depend continuously on the set $\sA$, and this was subsequently strengthened to Lipschitz continuity when $\sA$ does not admit a common invariant subspace \cite{W,K10} or when $\sA$ admits a strictly invariant cone \cite{MW}. On the other hand, the lower spectral radius is in general only upper semi-continuous: see \cite[p.11--13, 20]{Jbook} or the discussion below.

The principal aim of this article is to examine in detail the continuity properties of the lower spectral radius, giving in particular a sufficient condition for Lipschitz continuity of the lower spectral radius in the neighbourhood of particular sets of matrices. We then give a general necessary and sufficient condition for the lower spectral radius to be continuous at a given set of matrices. As a subsequent application of these results we exhibit open sets of $k$-tuples of $2 \times 2$ matrices within which the lower spectral radius generically fails to be realised as the spectral radius of a finite product of matrices; for related research in the context of the upper spectral radius we note \cite{BM,BTV,HMST,LW,MS}. 

\subsection{Initial observations on continuity and discontinuity}

Let $\cK(M_d(\R))$ denote the set of all compact nonempty subsets of $M_d(\R)$, which we equip with the Hausdorff metric defined by
\[d_H\left(\sA,\sB\right):=\max\left\{\sup_{A \in \sA}\inf_{B \in \sB} \|A-B\|, \sup_{B \in \sB}\inf_{A \in \sA} \|A-B\|\right\},\]
where $\|\mathord{\cdot}\|$ denotes the Euclidean operator norm.
With respect to this metric $\cK(M_d(\R))$ is a complete metric space. We also let $\cK(\GL_d(\R))$, $\cK(\SL_d(\R))$, et cetera denote the set of all compact nonempty subsets of $\GL_d(\R)$, $\SL_d(\R)$, and so forth, which we equip with the same metric. We use the notation $\GL_d^+(\R)$ to denote the set of invertible $d \times d$ matrices with positive determinant.

\medskip

Given $\sA \in \cK(M_d(\R))$, an important theorem of M.A.~Berger and Y.~Wang \cite{BW} asserts that the upper spectral radius may be written as
$$\jrad(\sA)= \inf_{n \geq 1}\sup\left\{\left\|A_n\cdots A_1\right\|^{\frac{1}{n}}\colon A \in \sA\right\}
=  \sup_{n \geq 1}\sup \left\{\rho\left(A_n\cdots A_1\right)^{\frac{1}{n}}\colon A \in \sA\right\}.$$
The first of these two alternative expressions is an infimum of a sequence of functions each of which depends continuously on $\sA$, and such an infimum is necessarily upper semi-continuous. Conversely, the second expression is a supremum of continuous functions of $\sA$ and hence is lower semi-continuous. It follows from the equality between these two quantities that the joint spectral radius is a continuous function from $\cK(M_d(\R))$ to $\R$. As was previously indicated this observation originates with Heil and Strang \cite{HS}.

In the case of the lower spectral radius one may in a related manner write
\begin{equation}\label{eq:subrad-inf}
\subrad(\sA)= \inf_{n \geq 1}\inf\left\{\left\|A_n\cdots A_1\right\|^{\frac{1}{n}}\colon A \in \sA\right\}
=  \inf_{n \geq 1}\inf \left\{\rho\left(A_n\cdots A_1\right)^{\frac{1}{n}}\colon A \in \sA\right\},
\end{equation}
see for example \cite[p.11--13]{Jbook}. Crucially however this expresses $\subrad$ only as an infimum of continuous functions and not also as a supremum, so the upper semi-continuity of the lower spectral radius is guaranteed by this expression but its lower semi-continuity is not. R.M.~Jungers (\cite[p.20]{Jbook}) has previously noted an example involving non-invertible matrices where the lower spectral radius fails to be lower semi-continuous. In this article we will find the following very simple example to be particularly instructive:
\begin{example}\label{ex:simple}
Define
\[\sA:=\left\{\left(\begin{array}{cc}2&0\\0&\frac{1}{8}\end{array}\right),\left(\begin{array}{cc}1&0\\0&1\end{array}\right)\right\}.\]
Then the lower spectral radius is discontinuous at $\sA$. Specifically, if for each $n \geq 1$ we define
\[\sA_n:=\left\{\left(\begin{array}{cc}2&0\\0&\frac{1}{8}\end{array}\right),\left(\begin{array}{cc}\cos\frac{\pi}{2n}&-\sin\frac{\pi}{2n}\\\sin\frac{\pi}{2n}&\cos\frac{\pi}{2n}\end{array}\right)\right\}\]
then $\lim_{n \to \infty}\sA_n = \sA$, $\subrad(\sA)=1$, and $\subrad(\sA_n)= \frac{1}{2}$ for every $n \geq 1$.

Let us briefly justify these claims. It is clear that $\subrad(\sA)=1$ and that $\sA_n \to \sA$, so we must show that $\subrad(\sA_n)\equiv \frac{1}{2}$. Given $n \geq 1$ it is clear that for any product of elements of $\sA_n$, the spectral radius of that product is at least the square root of its determinant. Since the minimum of the determinants of the two matrices is $\frac{1}{4}$ it follows easily that $\subrad(\sA_n) \geq \frac{1}{2}$. On the other hand for each $m,n \geq 1$ we have
\[
\left(\begin{array}{cc}2&0\\0&\frac{1}{8}\end{array}\right)^m 
\left(\begin{array}{cc}\cos\frac{\pi}{2n}&-\sin\frac{\pi}{2n}\\\sin\frac{\pi}{2n}&\cos\frac{\pi}{2n}\end{array}\right)^n
= \left(\begin{array}{cc}0&-2^m\\\frac{1}{8^m}&0\end{array}\right)
\]
which has spectral radius $\frac{1}{2^m}$. Since $m$ may be taken arbitrarily large while $n$ remains fixed we deduce with the aid of \eqref{eq:subrad-inf} that ${\subrad(\sA_n)\equiv\frac{1}{2}}$. 
\end{example}

The above example suggests the following general mechanism for constructing discontinuities of the lower spectral radius. Given a set $\sA \in \cK(\GL_2(\R))$ it is clear that there exists an element $A_0$ of $\sA$ which minimises the absolute value of the determinant. Let us suppose furthermore that the lower spectral radius of $\sA$ is strictly greater than $|\det A_0|^{1/2}$. For this to be the case it is necessary that $A_0$ has distinct real eigenvalues. If there exists an additional element $B$ of $\sA$ whose eigenvalues agree or form a conjugate pair, then by applying a precise but arbitrarily small perturbation to $B$ we may arrange that some large power of $B$ maps the more expanding eigenspace of $A_0$ onto the more contacting eigenspace of $A_0$. By composing this power of $B$ with an even larger power of $A_0$ we may obtain long products whose spectral radii (when normalised for the length of the product) closely approximate $|\det A_0|^{1/2}<\subrad(\sA)$. In this manner we may obtain arbitrarily small perturbations of $\sA$ whose lower spectral radius is less than that of $\sA$ by a fixed amount.

In fact the condition that $\sA$ should contain a matrix $B$ whose eigenvalues are equal in modulus is much stronger than is required. In order to formulate our result we require the notion of \emph{domination}.

\subsection{Domination}

The notion of domination originated in the theory of ordinary differential equations, where it is known as `exponential separation' (see e.g.\ \cite{Palmer} and references therein). It was rediscovered in differentiable dynamics, where it played an important role in the solution of the Palis--Smale $C^1$-stability conjecture (see \cite[Appendix~B]{BDV} and references therein). In ergodic theory, the concept of domination is related to continuity properties of Lyapunov exponents: \cite{BoV}. Domination is also relevant to control theory \cite[\S~5.2]{CK}. We will not give the most general definition of domination, but instead we will use the characterisations better adapted to our context, which come from \cite{BG}. We use the notation $\sigma_1(A),\ldots,\sigma_d(A)$ to denote the singular values of the matrix $A \in M_d(\R)$, which are the square roots of the eigenvalues of the positive semidefinite matrix $A^*A$ listed in decreasing order according to multiplicity.

\begin{definition}\label{de:dominated}
Let $\sA \in \cK(\GL_d(\R))$ and suppose that $1 \leq k < d$. We say that $\sA$ is \emph{$k$-dominated}, or that $k$ is an \emph{index of domination for $\sA$},  if one of the following equivalent conditions holds. Either:
\begin{enumerate}[(a)]
\item\label{i:dom-singular}
There exist constants $C>1$, $\tau \in (0,1)$ such that
$$
\frac{\sigma_{k+1}(A_n \dots A_1)}{\sigma_{k}(A_n \dots A_1)} \leq C \tau^n
\quad \forall n \ge 1, \ \forall A_1, \dots, A_n \in \sA,
or
$$
\item\label{i:dom-multicone}
There exists a set $\cC \subset \R^d \setminus \{0\}$ with the following properties:
	\begin{enumerate}[(i)]
		\item $\cC$ is relatively closed in $\R^d\setminus \{0\}$;
		\item $\cC$ is homogeneous (i.e., closed under multiplication by nonzero scalars);
		\item the image set $\sA \cC := \bigcup_{A\in \sA} A(\cC)$ is contained in the interior of $\cC$;
		\item there exists a $k$-dimensional subspace of $\R^d$ which is contained in $\cC \cup\{0\}$;
		\item there exists a $(d-k)$-dimensional subspace of $\R^d$ which does not intersect any element of $\cC$.
	\end{enumerate}
\end{enumerate}
When (\ref{i:dom-multicone}) holds the set $\cC$ is called a \emph{$k$-multicone} for $\sA$. 

We shall also say that every set $\sA \in \cK(\GL_d(\R))$ is $d$-dominated.
\end{definition}

The equivalence of the different formulations of domination listed above was proved in \cite[Theorem~B]{BG}. 
Structures similar to multicones have been studied previously in the context of the action of a single matrix or operator (see e.g. \cite{KS,KrLiSo}) but in the context of several matrices or operators acting simultaneously their use appears to be quite recent, originating in \cite{ABY} in the context of two-dimensional matrices. Note that if $\cC$ is a $k$-multicone for a given set $\sA$ then it is also a $k$-multicone for every sufficiently nearby set $\sB$, and therefore for $1 \leq k < d$ each of the sets
\[\left\{\sA \in \cK\left(\GL_d\left(\R\right)\right)\colon \sA\text{ is }k\text{-dominated}\right\}\]
is open. 

The simplest example of domination is perhaps the following: if $A$ is a matrix in $\GL_d(\R)$
with eigenvalues $\mu_1$, \dots, $\mu_d$ repeated according to multiplicity and
ordered as $|\mu_1|\ge \cdots \ge |\mu_d|$, then for any $k$ such that $|\mu_k| > |\mu_{k+1}|$, the singleton set $\sA := \{A\}$ is $k$-dominated. By the remark above, every set $\sB$ which is sufficiently close to $\sA$ is also $k$-dominated. In this case the $k$-multicone $\cC$ can be taken as a single cone around an eigenspace of $A$. In general multicones may have arbitrarily many connected components,
as the following example indicates:
\begin{example}
Let $1 \le k < d$.
Suppose $E_1$, \dots, $E_m \subset \R^d$ are $k$-dimensional subspaces
and $F_1$, \dots, $F_n \subset \R^d$ are $(d-k)$-dimensional subspaces such that
$E_i \cap F_j = \{0\}$ for every $(i,j) \in I:=\{1,\dots,m\}\times\{1,\dots,n\}$.
If $\lambda > 1$ is sufficiently large
and for each $(i,j)\in I$ we choose a matrix $A_{i,j} \in \GL_d(\R)$
that fixes the spaces $E_i$ and $F_j$ and satisfies 
$$
\big\| A_{i,j}^{-1}| E_i \big\|^{-1} > \lambda \big\| A_{i,j}| F_j \big\| \, ,
$$
then the set $\sA$ composed of those matrices is $k$-dominated.
Indeed, we can take a $k$-multicone $\cC$ such that
$\cC \cup \{0\}$ is the union of all $k$-dimensional subspaces 
sufficiently close to some $E_i$.
Also notice that any $k$-multicone for $\sA$ contains $E_1 \cup \cdots \cup E_n \setminus \{0\}$
and does not intersect $F_1 \cup \cdots \cup F_n$.
\end{example}

In the example above, the spaces $E_i$'s and $F_j$'s can be intertwined in complex ways, preventing the existence of a topologically simple $k$-multicone: see \cite[\S~4]{BG}
for some peculiarities which may arise in higher dimensions. Multicones can be surprisingly intricate even if $\sA$ is a pair of matrices in $\SL_2(\R)$: some relevant examples appear in Figures~1 and 2 in the paper \cite{ABY},
and a complete description of the possible multicones of pairs of $\SL_2(\R)$-matrices is obtained in \S{3.8} of that paper. This description shows in particular that for every integer $n \geq 1$, there exists a $1$-dominated pair of $\SL_2(\R)$-matrices for which every $1$-multicone has at least $n$ connected components.

\medskip

Before proceeding further let us recall some facts and notation from multilinear algebra, which we summarise here for the convenience of the reader. We recall (see e.g.\ \cite[p.557]{MB}) that the  Euclidean inner product $\langle \cdot,\cdot\rangle$ on $\R^d$ induces a natural inner product on the $d \choose k$-dimensional vector space $\wed^k\R^d$, which on pairs of decomposable vectors is given by:
\begin{equation}\label{eq:wedge-inner-product-definition}\langle u_1 \wedge \cdots \wedge u_d,v_1 \wedge \cdots \wedge v_d\rangle := \det \left(\left[\langle u_i,v_j\rangle\right]_{i,j=1}^d\right).\end{equation}
If $A$ is a matrix in $M_d(\R)$, which we identify with a linear map $A \colon \R^d \to \R^d$,
we let $\wed^k A \colon \wed^k\R^d \to \wed^k\R^d$ denote its $k^\mathrm{th}$ exterior power. 
Let us list some properties of this linear map,
referring the reader to \cite[p.119--120]{Arnold} for details.
If $\mu_1(A)$, \dots, $\mu_d(A)$ are the eigenvalues of $A$,
repeated according to multiplicity, then 
the eigenvalues of $\wed^k A$ are the numbers
$$
\mu_{i_1}(A) \mu_{i_2}(A) \cdots \mu_{i_k}(A), \quad 
\text{where } i_1 < i_2 < \cdots < i_k \, ,
$$
and the repetitions in this list match the multiplicities.
In particular, if we order the eigenvalues as  $|\mu_1(A)| \ge \cdots \ge |\mu_d(A)|$
then 
\begin{equation}\label{eq:power-rho}
\rho(\wed^k A) = |\mu_1(A)| \cdots |\mu_k(A)| \, .
\end{equation}
Similarly, the list of singular values of $\wed^k A$ (with respect to the inner product \eqref{eq:wedge-inner-product-definition}) with repetitions according to multiplicity is exactly:
$$
\sigma_{i_1}(A) \sigma_{i_2}(A) \cdots \sigma_{i_k}(A), \quad 
\text{where } i_1 < i_2 < \cdots < i_k \, .
$$
In particular
\begin{align}
\|\wed^k A\|   = 
\sigma_1 (\wed^k A) &= \sigma_1(A) \cdots \sigma_{k-1}(A) \sigma_k(A) , \label{eq:power-norm} 
\\
\sigma_2 (\wed^k A) &= \sigma_1(A) \cdots \sigma_{k-1}(A) \sigma_{k+1}(A) .\label{eq:power-sigma2}
\end{align}
Let us also notice that the classical fact that
\begin{equation}\label{eq:super-Gelfand}
\lim_{n \to \infty} \sigma_k(A^n)^{1/n}=|\mu_k(A)|
\end{equation}
can be readily deduced by applying Gelfand's formula to $\wed^k A$.
We shall use these identities frequently without comment.

If $\sA$ is a subset of $M_d(\R)$, we denote by $\wed^k \sA$ the set of all $\wed^k A$ with $A \in \sA$. Returning to the subject of domination, we observe that when $1 \leq k \leq d$ the set $\sA \in \cK(\GL_d(\R))$ is $k$-dominated if and only if the set $\wed^k \sA$ is $1$-dominated: 
this follows from Definition~\ref{de:dominated}(\ref{i:dom-singular})
together with formulas \eqref{eq:power-norm} and \eqref{eq:power-sigma2}.

\subsection{The main results}

Our first result substantially extends the principles behind Example~\ref{ex:simple}. We are able to show that for sets of $2$-dimensional matrices with strictly positive determinants, we may obtain discontinuities of the lower spectral radius by composing elements of the set with rotations in considerable generality. Here and throughout this article we let $R_\theta \in SO(2)$ denote anticlockwise rotation of the plane through angle $\theta$.
\begin{theorem}\label{th:GL2+}
Suppose that $\sA \in \cK(\GL_2^+(\R))$ satisfies
$$
\subrad(\sA) > \subrad(\wed^2 \sA)^{1/2}
$$
and $\sA$ is not $1$-dominated. Then the function $\theta \in \R \mapsto \subrad(R_\theta \sA) \in \R_+$ is discontinuous at $\theta = 0$.
\end{theorem}
Note that since $\wed^2\R^2$ is one-dimensional, the right-hand side in the inequality above is simply
\begin{equation}\label{eq:min-det}
\subrad(\wed^2 \sA)^{1/2} = \inf_{A \in \sA} \left( \det A \right)^{1/2}.
\end{equation}

In the converse direction to Theorem~\ref{th:GL2+} it is natural to ask how this mechanism for creating discontinuities might be avoided. A particularly simple possibility is that in which $\sA$ consists of a single matrix with unequal real eigenvalues: since no product of elements of $\sA$ can have almost-equal eigenvalues, perturbations of the type underlying Example~\ref{ex:simple} cannot be performed. Similarly, if $\sA$ consists of many different elements, but all of those elements have distinct eigenvalues, their expanding directions are all closely aligned to one another, and their contracting directions are also all closely aligned to one another, then it is difficult to see how a discontinuity might be constructed along the lines of Example~\ref{ex:simple}. It transpires that $1$-domination is the appropriate condition to prevent perturbations which discontinuously reduce the lower spectral radius:
\begin{theorem}\label{th:lipschitz}
Let $\mathcal{D}\subset \cK(\GL_d(\R))$ denote the open set of all nonempty compact subsets of $\GL_d(\R)$ which are $1$-dominated. Then $\subrad \colon \mathcal{D} \to \R$ is locally Lipschitz continuous.
\end{theorem}
The proof of this theorem is in the same spirit as the corresponding Lipschitz continuity results for the upper spectral radius. These theorems (in \cite{K10,MW,W}) prove that $\jrad$ is Lipschitz continuous near $\sA$ by demonstrating the existence of an operator norm $|\mathord{\cdot}|_{\sA}$ on $M_d(\R)$ such that $|A|_{\sA}\leq \jrad(\sA)$ for all $A \in \sA$. It follows easily that if $d_H(\sA,\sB)<\epsilon$ then $|B|_\sA\leq |A|_\sA+K\epsilon \leq \jrad(\sA)+K\epsilon$ for all $B \in \sB$ and therefore $\jrad(\sB)\leq \jrad(\sA)+K\epsilon$, where $K$ is a constant related to the eccentricity of the norm $|\mathord{\cdot}|_{\sA}$. If the norm $|\mathord{\cdot}|_{\sA}$ may be chosen in a systematic way so that its eccentricity depends continuously on $\sA$ then by proving the same inequality with $\sA$ and $\sB$ interchanged it follows that $\jrad$ is Lipschitz continuous at $\sA$. In \cite{K10,W} this is achieved by taking $|\mathord{\cdot}|_{\sA}$ to be a \emph{Barabanov norm}, that is, the operator norm on $M_d(\R)$ induced by a norm $|\mathord{\cdot}|_{\sA}$ on $\R^d$ such that for every $v \in \R^d$
\[\jrad(\sA)|v|_{\sA}=\max_{A \in \sA} |Av|_{\sA}.\]
The existence of such norms was established by N.E.~Barabanov in \cite{Ba} and their properties were subsequently examined in detail by F.~Wirth \cite{W2}. In order to prove Theorem~\ref{th:lipschitz} we construct a function on the set $\cC$ considered in Definition~\ref{de:dominated}(\ref{i:dom-multicone}) which satisfies a functional equation similar to that of a Barabanov norm. An object satisfying a functional equation of this kind was used recently by N.~Guglielmi and V.Yu.~Protasov in \cite{GP} to estimate the lower spectral radii of sets of matrices which preserve an embedded pair of convex cones. Unlike that construction we make no use of convexity properties and do not require the preserved region $\cC$ to be connected: in fact our construction more closely resembles a method used by T. Bousch in ergodic optimisation \cite{Bou,Bou2}. For further details we direct the reader to Proposition~\ref{pr:barabanov} below. We remark that R.M.~Jungers has previously shown that the lower spectral radius is continuous at sets of matrices which preserve a nested pair of convex cones \cite{J12}.

\medskip

As well as giving us a sufficient condition for continuity of the lower spectral radius, Theorem~\ref{th:lipschitz} yields a nontrivial lower bound on the values taken by $\subrad(\sB)$ when $\sB$ is close to $\sA$. To see this we argue as follows. If $A \in \GL_d(\R)$ and $1 \leq k_1\leq k_2 \leq d$ then the formula \eqref{eq:power-rho} shows that $\rho(\wed^{k_1}A)^{1/k_1}\geq \rho(\wed^{k_2}A)^{1/k_2}$. Combining this inequality with the characterisation \eqref{eq:subrad-inf} of the lower spectral radius we obtain that for any $\sA \in \cK(\GL_d(\R))$,
\begin{equation}\label{eq:wedges}
	\subrad(\sA) \geq \subrad\left(\wed^2\sA\right)^{1/2}\geq \cdots \geq \subrad\left(\wed^{d-1}\sA\right)^{1/(d-1)}\geq \subrad\left(\wed^d\sA\right)^{1/d}.
\end{equation}
Let us  suppose now that $\sA \in \cK(\GL_d(\R))$ is $k$-dominated, in which case $\wed^k\sA$ is $1$-dominated. In view of Theorem~\ref{th:lipschitz} it must be the case that the map $\sB \mapsto \subrad(\wed^k \sB)$ is continuous at $\sA$, and in view of \eqref{eq:wedges} we deduce
\[\liminf_{\sB \to \sA} \subrad(\sB) \geq \lim_{\sB \to \sA}\subrad\left(\wed^k\sB\right)^{1/k}=\subrad\left(\wed^k\sA\right)^{1/k}.\]
Since this is true for every $k$ such that $\sA$ is $k$-dominated, it is in particular true for the least such integer $k$. If for every $\sA \in \cK(\GL_d(\R))$ we define
\[
\ell(\sA) := \min \big\{ k \in \{1,\dots,d\} \colon 
\text{$\sA$ is $k$-dominated} \big\},
\]
the we have proved
\begin{equation}\label{eq:iff-weaker}
\liminf_{\sB \to \sA} \subrad(\sB) \geq 
\subrad\left(\wed^{\ell(\sA)} \sA\right)^{1/\ell(\sA)} \, .
\end{equation}
The next key result of this article is the improvement of this inequality to an equality. We have:
\begin{theorem}\label{th:formula-GLd}
For every $\sA \in \cK(\GL_d(\R))$, 
\begin{equation}\label{eq:formula-GLd}
\liminf_{\sB \to \sA} \subrad(\sB) =
\subrad\left(\wed^{\ell(\sA)} \sA\right)^{1/\ell(\sA)} \, ,
\end{equation}
where $\ell(\sA)$ is the smallest index of domination for $\sA$.
\end{theorem}
As was previously remarked, the identity \eqref{eq:subrad-inf} implies that $\subrad \colon \cK(\GL_d(\R))\to \R_+$ is upper semi-continuous, so 
$$
\limsup_{\sB \to \sA} \subrad(\sB) = \subrad(\sA)
$$
for all $\sA \in \cK(\GL_d(\R))$
and therefore Theorem~\ref{th:formula-GLd} has the following consequence:
\begin{corollary}\label{co:iff-GLd}
The set $\sA \in \cK(\GL_d(\R))$ is a point of continuity of the function $\subrad \colon \cK(\GL_d(\R)) \to \R_+$ if and only if
\begin{equation}\label{eq:iff-continuity}
\subrad(\sA) = \subrad\left(\wed^{\ell(\sA)} \sA\right)^{1/\ell(\sA)} \, .
\end{equation}
\end{corollary}

Since the function $\sA \mapsto \ell(\sA)$ is upper semi-continuous, and the map $\sA \mapsto \subrad(\wed^k\sA)$ is continuous on $\{\sA \colon \ell(\sA) = k\}$ by Theorem~\ref{th:lipschitz}, it follows from \eqref{eq:wedges} that the map $\sA \mapsto \subrad(\sA)^{1/\ell(\sA)}$ is lower semi-continuous. Corollary~\ref{co:iff-GLd} thus has in common with the proof of the continuity of the upper spectral radius the fact that continuity is derived from an equality between an upper semi-continuous quantity and a lower semi-continuous one.

The arguments required to improve the inequality \eqref{eq:iff-weaker} to the equation \eqref{eq:iff-continuity} in Theorem~\ref{th:formula-GLd} are somewhat involved, and we defer describing them in detail until \S\ref{se:iff-GLd} below. However we remark that this part of the proof has in common with Example~\ref{ex:simple} and Theorem~\ref{th:GL2+} the core idea of finding a product $P$ of elements of $\sA$ and a product $R$ of matrices close to the set $\sA$ such that $R$ interposes two carefully-chosen subspaces of $\R^d$ associated to the spectrum of $P$ in such a manner that certain singular values of $PRP$ are brought into greater agreement with one another. In common with the proof of Corollary~\ref{co:iff-GLd}, our proof of Theorem~\ref{th:GL2+} in fact proceeds by deriving the formula
\begin{equation}\label{eq:formula-GL2+}
\liminf_{\theta \to 0} \subrad(R_\theta \sA) =
\subrad\left(\wed^{\ell(\sA)} \sA\right)^{1/\ell(\sA)} \, ,
\end{equation}
for arbitrary $\sA \in \cK(\GL_2^+(\R))$. However, our proof of Theorem~\ref{th:formula-GLd} differs from 
Theorem~\ref{th:GL2+} and Example~\ref{ex:simple} in that the sets $\sB$ which are chosen close to $\sA$ when calculating the limit inferior may have strictly larger cardinality than $\sA$. At this time we are not able to prove a version of formula~\eqref{eq:formula-GLd} in which the cardinality of the set $\sA$ is maintained. On the other hand we have no reason to believe that such a result should be impossible; for further discussion see \S\ref{se:future} below.

We judge the following implications of Corollary~\ref{co:iff-GLd} to be sufficiently interesting to be worthwhile stating explicitly:
\begin{corollary}
Let $\sA \in \cK(\GL_d(\R))$ be a singleton set. Then $\subrad$ is continuous at $\sA$.
\end{corollary}
\begin{proof}
Let $\sA=\{A\}$. 
It follows from identities \eqref{eq:power-norm}--\eqref{eq:super-Gelfand} and Definition~\ref{de:dominated}(\ref{i:dom-singular}) that $\ell(\sA)$ is the sum of the algebraic multiplicities of those eigenvalues of $A$ whose moduli equal $\rho(A)$. 
So we obtain from \eqref{eq:power-rho} that $\rho(A)=\rho(\wed^{\ell(\sA)}A)^{1/\ell(\sA)}$,
which is trivially equivalent to \eqref{eq:iff-continuity},
thus establishing the announced continuity.
\end{proof}

\begin{corollary}
The lower spectral radius $\subrad \colon \cK(\GL_2(\R)) \to \R_+$ is continuous at every $\sA \in \cK(\SL_2(\R))$. In particular the restriction $\subrad \colon \cK(\SL_2(\R)) \to \R_+$ is continuous.
\end{corollary}
\begin{proof}
Let $\sA \in \cK(\SL_2(\R))$, and note that $\wed^2\sA=\{1\}$ so that trivially $\subrad(\wed^2\sA)^{1/2}=1$.
Since $\sigma_2(A)=\sigma_1(A)^{-1}$ for all $A \in \SL_2(\R)$, direct examination of Definition~\ref{de:dominated}(\ref{i:dom-singular}) shows that $\sA$ is $1$-dominated
if and only if $\subrad(\sA)>1$.
It follows from these facts that \eqref{eq:iff-continuity} always holds.
\end{proof}

So far in this introduction we have presented only one example -- Example~\ref{ex:simple} -- of a set of matrices at which the lower spectral radius is discontinuous. We will present later, in \S\ref{se:examples} below, a systematic method for constructing pairs of $\GL_2(\R)$-matrices which form discontinuities of $\subrad$. We also indicate a higher-dimensional generalisation of Example~\ref{ex:simple} in that section. We remark that all of these examples are somewhat `rigid' in the sense that they consist of sets of matrices which simultaneously preserve a finite union of subspaces of $\R^d$: however, we believe that this kind of rigidity is not a necessary condition for discontinuity, and that the set of discontinuities may even be quite large in certain senses (see \S\ref{se:future} below).


\subsection{Finiteness properties and computation}

In the influential article \cite{LW} J.C.~Lagarias and Y.~Wang conjectured that for every nonempty finite set $\sA \subset M_d(\R)$ there exists a finite sequence $A_1,\ldots,A_n$ of elements of $\sA$ such that $\rho(A_n\cdots A_1)^{1/n}=\jrad(\sA)$. A set $\sA$ for which this property holds is said to have the \emph{finiteness property} for the upper spectral radius; we will abbreviate this by saying that $\sA$ has the \emph{upper finiteness property}. Whilst Lagarias and Wang's conjecture was subsequently shown to be false \cite{BM}, it is believed that the upper finiteness property holds for \emph{typical} finite sets of matrices in various senses \cite{CGSZ,JB,Ma} and this question continues to stimulate research \cite{BTV,DHX13,HMST,MS}. In particular some of the most powerful methods for computing the upper spectral radius consist in verifying that the upper finiteness property is valid for a particular set of matrices and then identifying a finite sequence which attains the upper spectral radius in this manner (see for example \cite{GP,GZ,Ma}). By analogy let us say that $\sA$ satisfies the \emph{finiteness property for the lower spectral radius}, or \emph{lower finiteness property}, if there exists a finite sequence $A_1,\ldots,A_n \in \sA$ such that $\rho(A_n\cdots A_1)^{1/n}=\subrad(\sA)$. In view of Gelfand's formula we note that $\sA$ satisfies the lower finiteness property if and only if there exists a periodic infinite sequence $(A_i) \in \sA^\N$ such that $\lim_{n \to \infty} \|A_n\cdots A_1\|^{1/n}=\subrad(\sA)$. In this subsection we note an implication of our results for the lower finiteness property and discuss its consequences for computation of the lower spectral radius.

The following set of matrices was noted by T.~Bousch and J.~Mairesse \cite{BM} as an example where the lower finiteness property fails to hold:
\begin{example}\label{ex:nasty1}
Consider the set $\sA \subset \GL_2(\R)$ defined by
\[\sA:=\left\{\left(\begin{array}{cc}\frac{1}{3}&0\\0&3\end{array}\right),\left(\begin{array}{cc}2&0\\0&\frac{1}{2}\end{array}\right)\right\}.\]
Any product of elements of $\sA$ has spectral radius $\max\{3^a2^{-b},3^{-a}2^b\}$ where $a$ is the number of occurences of the first matrix and $b$ the number of occurrences of the second, and since $a/b$ may be arbitrarily close to $\log_32$ it follows that $\subrad(\sA)=1$. Conversely by the fundamental theorem of arithmetic the spectral radius of a finite product can never equal $1$.
\end{example}
The reader, noticing the quite degenerate structure of this pair of matrices, might hope that the lower finiteness property should hold at least for \emph{generic} finite sets of matrices. This is in fact quite untrue. We note the following consequence of Theorem~\ref{th:GL2+} and Corollary~\ref{co:iff-GLd}. By abuse of notation we define the lower spectral radius of a $k$-tuple $\sA \in \GL_d(\R)^k$ to be the lower spectral radius of the associated set, and similarly we say that a $k$-tuple is $\ell$-dominated if and only if the associated set is.

\begin{theorem}\label{th:no-min}
Let $k \geq 2$ and define $\cU \subset \GL_2^+(\R)^k$ 
as the largest open set such that all $\sA=(A_1, \dots, A_k)\in \cU$ satisfies the following conditions:
\begin{itemize}
\item
$\det A_1 < \det A_j$ for all $j = 2, 3, \dots, k$.
\item
The matrix $A_1$ has distinct real eigenvalues.
\item
$\sA$ is not $1$-dominated.
\end{itemize}
Then 
the set
\[
\cR:=\left\{\sA \in \cU \colon \subrad(\sA)=\subrad\left(\wed^2\sA\right)^{1/2}\right\}
\]
is a dense $G_\delta$ subset of $\cU$, and every $\sA \in \cR$ fails to have the lower finiteness property.
\end{theorem}

\begin{proof}
By definition the set $\cU$ is open,
and it function $\ell$ (the least index of domination) is constant equal to $2$ there.
It follows from Corollary~\ref{co:iff-GLd} that 
the set $\cR$ is thus precisely the set of points of continuity of $\subrad$ in $\cU$, and since $\subrad$ is an upper semi-continuous function and $\cU$ is a Baire space this set is a dense $G_\delta$. Now fix any $\sA \in \cR$, and consider a finite product $A_{i_n} \cdots A_{i_1}$ of matrices in $\sA$. If every $i_j$ is equal to $1$ then we have $\rho(A_{i_n}\cdots A_{i_1})^{1/n}=\rho(A_1)>(\det A_1)^{1/2}$. If not then $\rho(A_{i_n}\cdots A_{i_1})^{1/n} \geq (\det A_{i_n}\cdots A_{i_1})^{1/2n} > (\det A_1)^{1/2}$. In either case we have $\rho(A_{i_n}\cdots A_{i_1})^{1/n} > (\det A_1)^{1/2} = \subrad(\wedge^2 \sA)^{1/2} = \subrad(\sA)$. Since the product $A_{i_n} \cdots A_{i_1}$ is arbitrary, we have proved that $\sA$ does not have the lower finiteness property.
\end{proof}

\begin{remark}
The set $\cU$ is nonempty: indeed a sufficient condition for the failure of $1$-domination which is satisfied on an open set is that the semigroup generated by $\sA$ should include a matrix with non-real eigenvalues. The set $\cU$ for $k=2$ thus in particular includes Example~\ref{ex:simple}.
\end{remark}

\begin{remark}
One may extend Theorem~\ref{th:no-min} to show not only that every $\sA \in \cR$ fails to satisfy the lower finiteness property -- and hence there is no periodic sequence $(i_j) \in \{1,\dots,k\}^\N$ such that $\lim_{n \to \infty}\|A_{i_n}\cdots A_{i_1}\|^{1/n} = \subrad(\sA)$ -- but moreover for each $\sA \in \cR$ there is no ergodic shift-invariant measure on $\{1,\dots,k\}^\N$ with respect to which $\|A_{i_n} \cdots A_{i_1} \|^{1/n} \to \subrad(\sA)$ almost everywhere. This contrasts sharply with the situation for the upper spectral radius, where an ergodic measure with the analogous property always exists \cite{M13}. Since we make no use of ergodic theory in this article it would be digressive for us to introduce the concepts required to prove this statement. Interested readers who are already familiar with ergodic theory should anyway not have difficulty in modifying the proof of Theorem \ref{th:no-min} in this direction.
\end{remark}

At the time of writing very few algorithms for the computation of the lower spectral radius have been proposed; we are aware only of \cite{GP,PJB}. These algorithms both operate in the context where $\sA$ preserves a nested pair of invariant cones, which implies that $\sA$ is a limit of $1$-dominated sets (at least when $\sA$ consists of invertible matrices). The Guglielmi-Protasov algorithm for the computation of the lower spectral radius produces an exact result in the case where the set of matrices being examined satisfies the lower finiteness property. In the case where an invariant cone is strictly preserved it does not seem unreasonable to us to believe that this condition might be satisfied generically. Once the context of $1$-dominated sets is left behind, however, Theorem~\ref{th:no-min} shows that algorithms which depend on the lower finiteness property cannot directly succeed in computing $\subrad(\sA)$ for typical finite sets of matrices.

These observations however do not constitute immediate grounds for pessimism. In the situation where $\sA \in \cK(\GL_2^+(\R))$ belongs to the above set $\cR$ it is clearly problematic to compute $\subrad(\sA)$ via the norms or spectral radii of products of elements of $\sA$: yet by definition $\cR$ is precisely the set on which $\subrad(\sA)=\subrad(\wed^2\sA)^{1/2}$, and the latter quantity is trivial to compute, being precisely the square root of the minimum of the determinants of the elements of $\sA$. More generally, if it could be shown that the relation
\begin{equation}\label{eq:generic-equals}\subrad(\sA)=\subrad\left(\wed^{\ell(\sA)}\sA\right)^{1/\ell(\sA)}\end{equation}
held generically for finite sets $\sA \subset \GL_d(\R)$ then the problem of computing $\subrad$ for a generic finite set of matrices would reduce to the problem of computing the lower spectral radius of a $1$-dominated set, since the set $\wed^{\ell(\sA)}\sA$ on the right-hand side is necessarily $1$-dominated. 

In order to implement such a programme for computing the lower spectral radius it would of course be necessary to be able to determine algorithmically whether or not a finite set of $d \times d$ matrices is $k$-dominated, and also to give explicit examples of multicones for sets of matrices which are known to be $k$-dominated. So far it is known that for pairs of $\SL_2(\R)$-matrices (and hence for pairs of matrices with positive determinant) there exists a terminating algorithm for determining whether or not that pair is $1$-dominated (\cite[Remark~3.14]{ABY}) and it is also possible to construct $1$-multicones for $1$-dominated pairs in an explicit manner (\cite[\S{3.8}]{ABY}). However, even for triples of $\SL_2(\R)$-matrices this problem remains open, being hampered by the much more complicated topology of the set of $1$-dominated tuples in $\SL_2(\R)^3$ (compare \cite[Proposition 4.18]{ABY} with \cite[Theorem 3.2]{ABY}).

On a more pessimistic note, while we have some hope that \eqref{eq:generic-equals} should hold generically in the topological sense for finite sets $\sA\subset \GL_d(\R)$, the question of whether that relation should be expected to hold for finite subsets of $\GL_d(\R)$ which are typical in the sense of Lebesgue measure on $\GL_d(\R)^k$ may have an entirely different answer. We discuss both of these questions further in \S\ref{se:future} below.

\medskip

The observant reader will have noticed that whereas the results which we quote for the upper spectral radius refer to subsets of $M_d(\R)$, the results which we prove in this article refer only to subsets of $\GL_d(\R)$. The reason for this difference is that at the present time we do not have a satisfactory definition of what it means for a subset of $M_d(\R)$ to be $k$-dominated. This problem is also discussed further in \S\ref{se:future}.

\subsection{Organisation of the article}
The remainder of this article is structured as follows. 
The Lipschitz continuity result, Theorem~\ref{th:lipschitz}, is proved in \S\ref{se:lipschitz};
the proof itself is independent of the rest of the paper.
In \S\ref{se:access} we prove certain preliminary results that will be useful in the proofs of Theorems~\ref{th:GL2+} and \ref{th:formula-GLd}, and these are respectively given in \S\ref{se:GL2+} and \S\ref{se:iff-GLd}.
In \S\ref{se:examples} we exhibit further examples of the discontinuity of the lower spectral radius, and in \S\ref{se:future} we discuss some directions for future research.


\section{Lipschitz continuity} \label{se:lipschitz}

If $d=1$ then the elements of every $\sA \in \cK(\GL_d(\R))$ commute and it is not difficult to see that the lower spectral radius of $\sA$ is simply the minimum of the norms of its individual elements, which is trivially a Lipschitz continuous function of $\sA$. We shall therefore assume throughout this section that $d>1$. We will deduce Theorem~\ref{th:lipschitz} from the following result, the proof of which constitutes the principal content of this section.

\begin{proposition}\label{pr:lipschitz}
Let $\sK \in \cK(\GL_d(\R))$ be $1$-dominated. Then there exists a constant $K>0$ such that for every pair of nonempty compact sets $\sA, \sB \subseteq \sK$ we have $|\log \subrad(\sA)-\log\subrad(\sB)|\leq Kd_H(\sA,\sB)$.
\end{proposition}

Throughout this section we shall use the notation $P\R^d$ to denote the space of all $1$-dimensional subspaces of $\R^d$. We shall use the notation $\overline{u} \in P\R^d$ to denote the subspace generated by the nonzero vector $u \in \R^d$. We define a distance function on $P\R^d$ by 
\begin{equation}\label{eq:def_d}
d(\overline{u},\overline{v}) = \frac{\|u\wedge v\|} {\|u\|\,\|v\|} \, ,
\end{equation} 
which is clearly independent of the choice of representative vector $u \in \overline{u}$ and $v \in \overline{v}$. Using the definition \eqref{eq:wedge-inner-product-definition} of the inner product on $\wed^2\R^d$ we have
\[\|u \wed v\|^2 = \langle u \wed v, u \wed v\rangle = \langle u,u\rangle \langle v,v\rangle - \langle u,v\rangle \langle v,u\rangle = \|u\|^2 \|v\|^2\left(1-\cos^2 \angle (u,v)\right)\]
from which it follows that $d(\overline{u},\overline{v})$ is precisely the sine of the angle between the spaces $\overline{u}$ and $\overline{v}$. For a proof that $d$ is a metric on $P\R^d$ we refer the reader to \cite[p.121]{Arnold}.  

If $\cC$ is an homogeneous subset of $\R^d\setminus\{0\}$ -- that is, it is closed under multiplication by nonzero scalars -- then $\overline{\cC} := \{ \overline{u} \colon u \in \cC\}$ is a well-defined subset of $P\R^d$; moreover $\overline{\cC}$ uniquely determines $\cC$.
If $\cC$ is a multicone for some set of matrices in the sense of  Definition~\ref{de:dominated}(\ref{i:dom-multicone})
then we say that $\overline{\cC}$ is a \emph{projective multicone} for that set of matrices.
Before commencing the proof of Proposition~\ref{pr:lipschitz} we require some preliminary results, the majority of which relate to the action of a $1$-dominated set $\sK$ on a $1$-multicone $\cC \subset \R^d\setminus\{0\}$.

\subsection{Preliminary estimates}

\begin{lemma}\label{le:cone-basic}
Let $\cC_0$, $\cC \subset \R^d \setminus \{0\}$ be homogeneous sets such that 
$\overline{\cC_0}$ and $\overline{\cC}$ are closed in $P \R^d$, with $\overline{\cC_0} \subset \Int \overline{\cC}$, and suppose that there exists a $\left(d-1\right)$-dimensional subspace of $\R^d$ that does not intersect $\cC$. Then there exists a constant $\kappa_0>0$ such that for every linear map $A \colon \R^d \to \R^d$ that satisfies $A \cC \subseteq \cC_0$, and every vector $v \in \cC_0$, we have $\|Av\|\geq \kappa_0\|A\|\|v\|$. 
\end{lemma}

\begin{proof}
Consider the open set of all $\overline{u} \in P\R^d$ which do not intersect the projectivisation of the hypothesised $(d-1)$-dimensional subspace. 
The set $\overline{\cC_0}$ is compact and is contained in this set, so we may define a new inner product on $\R^d$ with respect to which the angle between any two elements of $\cC_0$ is less than, say, $\frac{\pi}{4}$. The resulting change of norm clearly has no effect on the validity of the lemma beyond altering the precise value of the constant $\kappa_0$. In particular we may freely assume that no two elements of $\cC_0$ are perpendicular. We will show that if the claimed constant $\kappa_0$ does not exist then this assumption is contradicted.

Let us assume for a contradiction that such a constant $\kappa_0$ does not exist. It follows that we may find a sequence of matrices $A_n$ each having norm $1$ and a sequence of unit vectors $u_n$ with $u_n \in \cC_0$  such that $A_n u_n \to 0$ and such that $A_n \cC \subseteq \cC_0$ for every $n \geq 1$. By passing to a subsequence we may assume that the sequences $A_n$ and $u_n$ converge to limits $A \in M_d(\R)$ and $u \in \cC_0$, and it is clear that these limit objects must satisfy $Au=0$.

Since $u$ is interior to $\cC$ we may find a constant $\delta\in(0,1)$ such that if $\|w\|=1$ and $\lambda \in [-2\delta,2\delta]$ then ${u+\lambda w} \in \cC$. Since the eigenspaces of $A_n^*A_n$ converge to those of $A^*A$ as $n \to \infty$, and $u$ belongs to the kernel of $A$, we may choose a sequence of unit vectors $v_n \in \R^d$ such that $A_n^*A_n v_n = \sigma_d(A_n)^2v_n$ for every $n \geq 1$ and such that $v_n \to u$. Fix $n$ large enough that $\|A_nv_n\| \leq \delta$, and also large enough that if $\|w\|=1$  and $\lambda \in [-\delta,\delta]$ then ${v_n+\lambda w} \in \cC$.

Choose now a vector $w_n\in \R^d$ such that $\|A_nw_n\|=\|w_n\|=1$.  Since $w_n$ and $v_n$ are eigenvectors of the symmetric matrix $A^*_nA_n$ which correspond to different eigenvalues, we have $\langle v_n ,w_n\rangle=0$ and therefore $\langle A_nv_n , A_nw_n\rangle = \langle v_n, A_n^*A_nw_n\rangle=0$. We have ${v_n\pm\lambda w_n} \in \cC$ for all $\lambda \in [-\delta,\delta]$, and hence in particular ${A_n(v_n\pm\lambda w_n)} \in {\cC_0}$ for all such $\lambda$. Given such a $\lambda$ we have
\[\langle A_n(v_n+\lambda w_n),A_n(v_n - \lambda w_n)\rangle = \|A_nv_n\|^2-\lambda^2.\]
Taking $\lambda =\|A_nv_n\| \in (0,\delta]$ we find that the two vectors $ A_n(v_n\pm\lambda w_n) \in {\cC_0}$ are mutually perpendicular, contradicting the choice of inner product made at the beginning of our argument. We deduce the existence of the claimed constant $\kappa_0$. 
\end{proof}

The following is a straightforward corollary of Lemma~\ref{le:cone-basic}:

\begin{lemma}\label{le:cone-vector}
Let $\sK \in \cK(\GL_d(\R))$ be $1$-dominated, and let $\cC \subset \R^d\setminus\{0\}$ be a $1$-multicone for $\sK$. 
Then there exists a constant $\kappa \in (0,1)$ such that for every $u \in \cC$ and every $A_1,\ldots,A_n \in \sK$, we have $\|A_n\cdots A_1u\|\geq \kappa\|A_n\cdots A_1\|\mathord{\cdot}\|u\|$.
\end{lemma}

We will find it useful to consider a modification of the `adapted metric' used in the proof of \cite[Theorem B]{BG}, which we construct below.

\begin{lemma}\label{le:metric}
Let $\sK \in \cK(\GL_d(\R))$ be $1$-dominated, and let $\overline{\cC} \subset P \R^d$ be a projective $1$-multicone for $\sK$.
Then there exist constants $\theta \in (0,1)$ and $C_2>1$
and a metric $d_\infty$ on $\overline{\cC}$ such that 
\begin{equation}\label{eq:contraction}	
d_\infty(\overline{Au},\overline{Av})\leq \theta d_\infty(\overline{u},\overline{v})
\end{equation} 
and 
\begin{equation}\label{eq:lip_equiv}
d\left(\overline{u},\overline{v}\right) \leq d_\infty\left(\overline{u},\overline{v}\right) \leq C_2 d\left(\overline{u},\overline{v}\right)
\end{equation}
for all $A \in \sK$ and $\overline{u},\overline{v} \in \overline{\cC}$.
\end{lemma}

Relation \eqref{eq:contraction} express that the adapted metric $d_\infty$
is uniformly contracted by the projective action of the matrices on $\sK$,
while \eqref{eq:lip_equiv} implies that the adapted metric is Lipschitz equivalent on $\overline{\cC}$
to the metric $d$ given by \eqref{eq:def_d}.

\begin{proof}
By Definition~\ref{de:dominated}(\ref{i:dom-singular}) there exist constants $C_1>1$ and $\tau \in (0,1)$ such that $\sigma_2(A_n\cdots A_1)\leq C_1\tau^n\sigma_1(A_n\cdots A_1)$ for all $A_1,\ldots,A_n \in \sK$. In this situation, for all $\overline{u},\overline{v}\in \overline{\cC}$ we have
\begin{align*}d\Big(\overline{A_n\cdots A_1u}, \ \overline{ A_n\cdots A_1v}\Big) &= \frac{\left\|A_n\cdots A_1 u \wedge A_n\cdots A_1v\right\|}{\left\|A_n \cdots A_1u\|\cdot \|A_n\cdots A_1v\right\|} \\
&\leq \frac{\sigma_1(A_n\cdots A_1)\sigma_2(A_n\cdots A_1)\|u\wedge v\|}{\|A_n\cdots A_1u\|\cdot \|A_n\cdots A_1v\|}\\
&\leq \frac{C_1\tau^n}{\kappa^2} \cdot \frac{\|u\wedge v\|}{\|u\|\mathord{\cdot}\|v\|} = C_1\kappa^{-2}\tau^n d\left(\overline{u},\overline{v}\right)\end{align*}
where $\kappa \in (0,1)$ is the constant provided by Lemma~\ref{le:cone-vector}. We may therefore define
\[
d_\infty\big(\overline{u},\overline{v}\big):=\sum_{n=0}^\infty \sup_{A_1,\ldots,A_n \in \sK} d\Big(\overline{A_n\cdots A_1 u}, \ \overline{A_n \cdots A_1 v}\Big)
\]
for every $\overline{u},\overline{v} \in \overline{\cC}$, and this sum is convergent and defines a metric on~ $\overline{\cC}$.
Fix  $\overline{u},\overline{v} \in \overline{\cC}$ and 
note that \eqref{eq:lip_equiv} holds with $C_2:= 1 + C_1\kappa^{-2}\tau /(1-\tau)$.
Moreover, for every $A \in \sK$,
\[
d_\infty\left(\overline{Au},\overline{Av}\right) \leq d_\infty\left(\overline{u},\overline{v}\right)-d\left(\overline{u},\overline{v}\right)
\le \theta d_\infty \left(\overline{u},\overline{v}\right)
\quad \text{where } \theta := 1- C_2^{-1} \, ,
\]
thus concluding the proof.
\end{proof}

We define a function $\varphi \colon \GL_d(\R) \times P\R^d \to \R$ as follows:
\begin{equation}\label{eq:def_varphi}
\varphi(A,\overline{u}) := \log \frac{\|Au\|}{\|u\|} \, .
\end{equation}

\begin{lemma}\label{le:varphi}
Let $\sK \in \cK(\GL_d(\R))$. 
Then there exists constants $C_3$, $C_4>0$ depending only on $\sK$ such that:
\begin{enumerate}[(a)]
\item\label{i:varphi_1}
for all $\overline{u} \in P\R^d$ the function $\varphi (\mathord{\cdot}, \overline{u})$ is $C_3$-Lipschitz continuous on $\sK$ with respect to the operator norm;
\item\label{i:varphi_2} 
for all $A \in \sK$ the function $\varphi(A,\mathord{\cdot})$ is $C_4$-Lipschitz continuous on $P\R^d$ with respect to the metric $d$ given by \eqref{eq:def_d}.
\end{enumerate}
\end{lemma}

\begin{proof}
Let 
$$
C_3 := \max_{A\in \sK} \|A^{-1}\| , \quad 
C_4 := \sqrt{2}\max_{A\in \sK} \|A^{-1}\|\,\|A\| .
$$
Let $A$, $B$ be matrices in $\sK$ and $u$, $v$ be unit vectors in $\R^d$.
If $\varphi(A,\overline{u}) \geq \varphi(B,\overline{u})$ then we may estimate
$$
\left|\varphi(A,\overline{u}) - \varphi(B,\overline{u})\right| = \log \frac{\|Au\|}{\|Bu\|} 
\leq \frac{\|Au\|}{\|Bu\|} - 1 \leq \frac{\|Au-Bu\|}{\|Bu\|}  \leq C_3 \|A-B\|,
$$
and if $\varphi(B,\overline{u}) \geq \varphi(A,\overline{u})$ the same result holds with a similar derivation. 
This proves part (\ref{i:varphi_1}).

Let $\alpha := \angle \left(\overline{u}, \overline{v}\right) \in [0, \pi/2]$,
so $d \left(\overline{u}, \overline{v}\right) = \sin \alpha$ and 
$$
\| u - v \| = \sqrt{2 - 2\cos \alpha} = \frac{\sqrt{2}}{\sqrt{1+\cos \alpha}} \, \sin \alpha
\le \sqrt{2} \, d \left(\overline{u}, \overline{v}\right).
$$
If $\varphi(A,\overline{u}) \geq \varphi(A,\overline{v})$ then we may estimate
$$
\left|\varphi(A,\overline{u}) - \varphi(A,\overline{v})\right| 
=   \log \frac{\|Au\|}{\|Av\|} 
\le \frac{\|Au\|}{\|Av\|} - 1 
\le \frac{\|A(u-v)\|}{\|Av\|} 
\le \|A\| \, \|A^{-1}\| \, \|u-v\|
\le  C_4  d \left(\overline{u}, \overline{v}\right).
$$
If $\varphi(A,\overline{v}) \geq \varphi(A,\overline{u})$ the same result holds with a similar derivation. 
\end{proof}

\subsection{Lower Barabanov functions}

The following key result gives us an analogue of the Barabanov norm -- introduced in \cite{Ba} for the study of the upper spectral radius -- for the lower spectral radius in the presence of $1$-domination. It is closely related to the \emph{Ma\~n\'e lemma} used by T.~Bousch and J.~Mairesse in \cite{BM} to study both the upper and the lower spectral radii of sets of matrices satisfying a positivity condition (see also  \cite[Lemme A]{Bou} for a related result) and to the concave `antinorm' used by N.~Guglielmi and V.~Yu.~Protasov in their algorithm for the computation of the lower spectral radius in the presence of an invariant convex cone \cite{GP}. Unlike Guglielmi and Protasov's constructions we make no use whatsoever of convexity properties. The version which we present here extends a result for $\GL_2(\R)$-matrices which was used previously by the first named author together with M.~Rams \cite{BRams}.


\begin{proposition}\label{pr:barabanov}
Suppose that $\sK \in \GL_d(\R)$ is $1$-dominated, and let $\cC$ be an associated $1$-multicone. Then for each nonempty compact set $\sB \subseteq \sK$ we may find a continuous function $\psi_{\sB} \colon \cC \to \mathbb{R}$ such that for every $u \in \cC$ and $t \in \R\setminus\{0\}$ we have:
\begin{equation}\label{eq:barabanov_main}
\psi_{\sB}(u) + \log\subrad(\sB) =  \min_{B \in \sB} \psi_{\sB}(Bu)
\end{equation}
and 
\begin{equation}\label{eq:barabanov_homog}
\psi_{\sB}(tu) = \psi_{\sB}(u) + \log |t|. 
\end{equation}
Moreover, there exists $C_5>0$ depending only on $\sK$ and $\cC$ such that the restriction of $\psi_{\sB}$ to unit vectors is $C_5$-Lipschitz:
\begin{equation}\label{eq:barabanov_Lip}
u, v \in \cC \text{ unit vectors } \Rightarrow \ 
| \psi_{\sB}(u) - \psi_{\sB}(v) | \leq C_5 d(\overline{u}, \overline{v}) .
\end{equation}
\end{proposition}

We call $\psi_{\sB}$ a \emph{lower Barabanov function}.

\begin{proof}
Lemma~\ref{le:metric} constructs an adapted metric $d_\infty$ on $\overline{\cC}$, with associated constant $\theta \in (0,1)$. Using Lemma~\ref{le:varphi} we may find a constant $C_4>0$ such that for all $B \in \sK$ the function $\varphi(B, \mathord{\cdot})$ is $C_4$-Lipschitz continuous with respect to the metric $d$; since by \eqref{eq:lip_equiv} $d \leq d_\infty$, these functions are also $C_4$-Lipschitz continuous on $\overline{\cC}$ with respect to the metric $d_\infty$. Let $W$ denote the set of all $f \colon \overline{\cC} \to\mathbb{R}$ which have Lipschitz constant less than or equal to $C_4/(1-\theta)$ with respect to this metric.  Clearly every $f \in W$ is $C_5$-Lipschitz with respect to the metric $d$ for some uniform constant $C_5>0$ which depends only on $\sK$. For each $f \in W$ let us define $Lf \colon \overline{\cC} \to \mathbb{R}$ by 
\[
(Lf)(\overline{u}):=\min_{B \in \sB} \left[ f(\overline{Bu})+\varphi(B,\overline{u})\right].
\] 
We give $W$ the metric induced by the supremum norm on $C(\overline{\cC};\R)$. We claim that $L$ acts continuously on $W$.

Let us first show that $L$ preserves $W$. If $f \in W$ and $\overline{u},\overline{v} \in \overline{\cC}$, choose $B_0 \in \sB$ such that $(Lf)(\overline{v})=f(\overline{B_0v})+\varphi(B_0,\overline{v})$. We have
\begin{align*}(Lf)(\overline{u})-(Lf)(\overline{v})&= \min_{B \in \sB} \left[f\left(\overline{Bu}\right)+\varphi\left(B,\overline{u}\right)\right] - f\left(\overline{B_0v}\right)-\varphi\left(B_0,\overline{v}\right)\\
&\leq f\left(\overline{B_0u}\right)-f\left(\overline{B_0v}\right) + \varphi\left(B_0,\overline{u}\right)-\varphi\left(B_0,\overline{v}\right)\\
&\leq \left(\frac{C_4}{1-\theta}\right)d_\infty\left(\overline{B_0u},\overline{B_0v}\right) + C_4d_\infty\left(\overline{u},\overline{v}\right)\\
& \leq \left(\frac{C_4}{1-\theta}\right)d_\infty\left(\overline{u},\overline{v}\right)\end{align*}
using the inequality $d_\infty\left(\overline{B_0u},\overline{B_0v}\right)\leq \theta d_\infty\left(\overline{u},\overline{v}\right)$,
and since $\overline{u},\overline{v} \in\overline{\cC}$ are arbitrary we have shown that $Lf \in W$. Let us now show that $L \colon W \to W$ is continuous. Let $f,g \in W$ and $\overline{u} \in \overline{\cC}$. Choosing similarly $B_0 \in \sB$ such that $(Lg)(\overline{u})=g(\overline{B_0u})+\varphi\left(B_0,\overline{u}\right)$ we may estimate
\begin{align*}(Lf)\left(\overline{u}\right)-(Lg)\left(\overline{u}\right) &=  \min_{B \in \sB} \left[f\left(\overline{Bu}\right)+\varphi\left(B,\overline{u}\right)\right] - g\left(\overline{B_0u}\right)-\varphi\left(B_0,\overline{u}\right)\\
&\leq f\left(\overline{B_0u}\right)-g\left(\overline{B_0v}\right)\leq  |f-g|_\infty.\end{align*}
Since $\overline{u} \in \overline{\cC}$ is arbitrary, by symmetry we obtain $|Lf-Lg|_\infty \leq |f-g|_\infty$ so that $L \colon W \to W$ is continuous as claimed.

Let $\hat{W}$ denote the set of equivalence classes of elements of $W$ modulo the addition of a real constant. It follows from the  Arzel\`a-Ascoli theorem that $\hat{W}$ is a compact subset of the Banach space $C(\overline{\cC};\R) \mod \R$, and it is not difficult to see that the function $L \colon W \to W$ induces a well-defined continuous transformation $\hat{L} \colon \hat{W} \to \hat{W}$. It follows by the Leray-Schauder fixed point theorem that there exists $\hat f_0 \in \hat{W}$ such that $\hat{L}\hat{f}_0=\hat{f}_0$, and consequently there exist $f_0 \in W$ and $\beta \in \R$ such that $Lf_0 = f_0 + \beta$.
Define
\[
\psi_{\sB} (u) := f_0( \overline{u} ) + \log \|u\| \quad \text{for all } u \in \cC.
\]
Note that $\psi_{\sB}$ has the desired properties \eqref{eq:barabanov_homog}, \eqref{eq:barabanov_Lip}.
Moreover,  for every $u \in \cC$, we have
\begin{equation}\label{eq:temp_1}
\psi_{\sB}(u) + \beta =  \min_{B \in \sB} \psi_{\sB}(Bu)
\end{equation}
and 
\begin{equation}\label{eq:temp_2}
\log\|u\| - C \le \psi_{\sB}(u) \le \log\|u\| + C,
\end{equation}
where $C:= |f_0|_\infty$.

To complete the proof of the lemma let us show that $\beta = \log\subrad(\sA)$.
Take a unit vector $u \in \cC$ 
Applying \eqref{eq:temp_1} recursively we obtain
\[
\min_{B_1,\ldots,B_n \in \sB} \psi_\sB (B_n\cdots B_1 u) = \psi_{\sB}(u) + n \beta.
\]
On the other hand, by \eqref{eq:temp_2} and Lemma~\ref{le:cone-vector},
for every $B_1,\ldots,B_n \in \sB$ we have
\[
\psi_\sB (B_n\cdots B_1 u) - C \le 
\log \| B_n \cdots B_1\|  \le \psi_\sB (B_n\cdots B_1 u) + C + \log \kappa^{-1} \, .
\]
Dividing by $n$, taking minimum over $B_1,\ldots,B_n \in \sB$, and making $n \to \infty$,
it follows that $\beta=\log\subrad(\sB)$, as claimed. The proof is complete.
\end{proof}

\begin{remark}
An almost identical construction can be applied to yield an `upper Barabanov function',
i.e., a function with the same properties as the lower Barabanov function $\psi_\sB$,
except that in \eqref{eq:barabanov_main} we replace $\subrad(\sB)$ with $\jrad(\sB)$ and $\min$ with $\max$.
\end{remark}

\subsection{Proof of Proposition~\ref{pr:lipschitz} and derivation of Theorem~\ref{th:lipschitz}}

\begin{proof}[of Proposition~\ref{pr:lipschitz}] 
Let $\sK \in \cK(\GL_d(\R))$ be $1$-dominated. 
Let $\cC$ be an associated $1$-multicone.
Let $\sA$, $\sB \subseteq \sK$ be nonempty compact sets, and let $\psi_{\sB} \colon \cC \to \R$ be a lower Barabanov function, as given by Proposition~\ref{pr:barabanov}.  

\begin{claim}
There exists $K>0$ depending only on $\sK$ and $\cC$ such that for any $A \in \sA$ and $u \in \cC$,
\begin{equation}\label{eq:lip4}
\psi_{\sB}(A u) \ge \psi_{\sB}(u) + \log \subrad(\sA) - K d_H(\sA,\sB) .
\end{equation}
\end{claim}

\begin{proof}[of the claim]
Given $A \in \sA$, choose $B \in \sB$ such that ${\|A-B\|} \leq d_H(\sA,\sB)$. 
Recall the definition \eqref{eq:def_varphi} of the function $\varphi$.
We have
\begin{equation}\label{eq:exterior}
d \left( \overline{Au}, \overline{Bu}\right) \le
\frac{\|Au\wedge Bu\|}{\|Au\|\mathord{\cdot}\|Bu\|}=\frac{\|Au \wedge (B-A)u\|}{\|Au\|\mathord{\cdot}\|Bu\|} \leq \frac{\|(B-A)u\|}{\|Bu\|}\leq C_6 \|A-B\|,
\end{equation}
where $C_6:=\max\{\|B^{-1}\|\colon B \in \sK\}$.
We then estimate:
\begin{align*}
|\psi_{\sB}(Au)& - \psi_{\sB}(Bu)| && \\
&\le \left| \psi_{\sB}\left(\tfrac{Au}{\|Au\|}\right) - \psi_{\sB}\left(\tfrac{Bu}{\|Bu\|}\right)\right|
+ \left| \varphi(A, \overline{u}) - \varphi(B, \overline{u}) \right|
\quad \text{(by \eqref{eq:barabanov_homog})} \\
&\le C_5 d \big( \overline{Au}, \overline{Bu} \big) + C_3 \|A-B\| 
\quad \text{(by \eqref{eq:barabanov_Lip} and Lemma~\ref{le:varphi}(\ref{i:varphi_1})}) \\
&\le (C_5 C_6 + C_3) \|A-B\| 
\qquad \text{(by \eqref{eq:exterior}.)} 
\end{align*}
So, letting $K:=C_5 C_6 + C_3$,
$$
\psi_{\sB}(A u) 
\ge \psi_{\sB}(B u) - K \|A-B\| 
\ge \psi_{\sB}(u) + \log \subrad(\sA) - K d_H(\sA,\sB) ,
$$
where in the last step we have used the main property \eqref{eq:barabanov_main}.
This proves the claim.
\end{proof}

Let us now fix $u\in \cC$.
Given $A_1,\ldots,A_n \in \sA$,
by iterating \eqref{eq:lip4} we obtain 
$$
\psi_{\sB} (A_n \cdots A_1 u) \ge \psi_{\sB}(u) + n \left[\log \subrad(\sA) - K d_H(\sA,\sB)\right].
$$
By the homogeneity property \eqref{eq:barabanov_homog},
we have
$$
\log \| A_n \cdots A_1 u \| \ge -C + \psi_{\sB} (A_n \cdots A_1 u), 
$$
where $C$ is the supremum of $|\psi_{\sB}|$ over unit vectors on $\cC$.
It follows that 
$$
\min_{A_1,\ldots,A_n \in \sA} \log \| A_n \cdots A_1 u \| \ge 
-C + \psi_{\sB}(u) + n \left[\log \subrad(\sA) - K d_H(\sA,\sB)\right].
$$
Dividing both sides by $n$ and letting $n \to \infty$ yields
$$
\log \subrad(\sA) \geq \log \subrad(\sB)-Kd_H(\sA,\sB).
$$

By a symmetrical argument we may also derive the reverse inequality $\log \subrad(\sB) \geq \log \subrad(\sA)-Kd_H(\sA,\sB)$, and this completes the proof of Proposition~\ref{pr:lipschitz}.
\end{proof}

We may now derive Theorem~\ref{th:lipschitz} from Proposition~\ref{pr:lipschitz}. Suppose that $\sA \in \cK(\GL_d(\R))$ is $1$-dominated: we wish to show that $\subrad$ is Lipschitz continuous on a neighbourhood of $\sA$. Since $1$-domination is an open property in $\cK(\GL_d(\R))$, we may find a constant $\epsilon>0$ such that if $\sB \in \cK(M_d(\R))$ and $d_H(\sA,\sB) \leq \epsilon$ then $\sB \in \cK(\GL_d(\R))$ and $\sB$ is $1$-dominated. Define
$$\sK:=\left\{A \in M_d(\R) \colon \exists\,B \in \sA \text{ such that } \|A-B\|\leq \epsilon\right\}.$$
Clearly $d_H(\sA,\sK)=\epsilon$ so that $\sK \in \cK(\GL_d(\R))$ and $\sK$ is $1$-dominated.

If $\sB_1$ and $\sB_2$ belong to the closed $\epsilon$-ball about $\sA$ in $\cK(\GL_d(\R))$ then it is clear from the definition of the Hausdorff metric that $\sB_1,\sB_2 \subseteq \sK$, and so by Proposition~\ref{pr:lipschitz} we have
\[\left|\log \subrad\left(\sB_1\right)-\log\subrad\left(\sB_2\right)\right| \leq Kd_H(\sB_1,\sB_1)\]
where $K>0$ depends only on $\sK$. Using the elementary real inequality $|e^x-e^y|\leq |x-y|e^{\max\{x,y\}}$ it follows that 
\[\left|\subrad\left(\sB_1\right)-\subrad\left(\sB_2\right)\right| \leq K\cdot\left(\sup\{\|A\|\colon A \in \sK\}\right) \cdot d_H(\sB_1,\sB_1).\]
We conclude that $\subrad$ is uniformly Lipschitz continuous on the closed $\epsilon$-neighbourhood of $\sA$, and this completes the derivation of Theorem~\ref{th:lipschitz}.

\section{Accessibility lemmas}\label{se:access}

As was indicated in the introduction, the proofs of both Theorem~\ref{th:GL2+} and Theorem~\ref{th:formula-GLd} rely on a mechanism whereby we transport a nonzero vector from one subspace of $\R^d$ to another by the action of a product of matrices each of which is close to the given set $\sA$. The requisite tools for this process are developed in this section.

\subsection{Two-dimensional accessibility}

In this subsection we wish to prove the following lemma which is used in the proof of Theorem~\ref{th:GL2+}, and from which we derive a higher-dimensional result which is used in the proof of Theorem~\ref{th:formula-GLd}. 

\begin{lemma}\label{le:access_2}
For all $\delta_1>0$ and $C>1$ there exist constants $c>0$ and $\lambda>1$ with the following property: 
if $A_1$, \ldots, $A_n$ are matrices in $\GL_2^+(\R)$ 
with $\|A_1\| \, \|A_1^{-1}\| \le C$
and
$v$, $w$ are nonzero vectors in $\R^2$ such that
\begin{equation}\label{eq:condition_v_w}
\frac{\|A_n \cdots A_1 v\|}{\|A_n \cdots A_1 w\|} \cdot \frac{\|w\|}{\|v\|} < c^2\lambda^{2n} ,
\end{equation}
then there exists ${\theta_1} \in [-\delta_1,\delta_1]$ such that the two vectors
$$
R_{\theta_1} A_n R_{\theta_1} A_{n-1} \cdots R_{\theta_1} A_1 v
\quad \text{and} \quad 
A_n A_{n-1} \cdots A_1 w
$$ 
are proportional to one another.
\end{lemma}

%

To begin the proof of Lemma~\ref{le:access_2} we require the following result:
\begin{lemma}\label{le:ABD}
For all $\delta_0>0$ there exist constants $c>0$ and $\lambda>1$ with the following property: 
if $A_1$, \ldots, $A_n$ are matrices in $\GL_2^+(\R)$ and $v$ is a vector in $\R^2$ such that 
\begin{equation}\label{eq:condition_ABD}
\|A_n \cdots A_1 v\| < c\lambda^n \|\wed^2 A_n \cdots A_1 \|^{1/2} \, \|v\|,
\end{equation}
then for any nonzero $u \in \R^2$ there exists $\theta_0 \in [-\delta_0,\delta_0]$ such that the vectors
$$
R_{\theta_0} A_n R_{\theta_0} A_{n-1} R_{\theta_0} \cdots R_{\theta_0} A_1 R_{\theta_0} v
\quad \text{and} \quad u
$$ 
are proportional to one another.
\end{lemma}
In the case that the matrices $A_i$ have determinant $1$ a proof of this statement (which uses the Hilbert projective metric) is given in \cite[Lemma C.2]{ABD}. 
The case of arbitrary positive determinant follows immediately. We deduce:
\begin{corollary}\label{co:ABD_bi}
For all $\delta_0>0$ there exist constants $c>0$ and $\lambda>1$ with the following property: 
if $A_1$, \ldots, $A_n$ are matrices in $\GL_2^+(\R)$ and 
$v$, $w$ are nonzero vectors in $\R^2$ satisfying condition \eqref{eq:condition_v_w}
then there exists ${\theta_0} \in [-\delta_0,\delta_0]$ such that the two vectors
$$
R_{\theta_0} A_n R_{\theta_0} A_{n-1} R_{\theta_0} \cdots R_{\theta_0} A_1 R_{\theta_0} v
\quad \text{and} \quad 
A_n A_{n-1} \cdots A_1 w
$$ 
are proportional to one another.
\end{corollary}

\begin{proof}
Given $\delta_0$,
let be $c$ and $\lambda$ be given by Lemma~\ref{le:ABD}.
Assume that $A_1$, \ldots, $A_n \in \GL_2^+(\R)$ and $v$, $w \in \R^2 \setminus \{0\}$
satisfy \eqref{eq:condition_v_w}.
Equivalently, we have  $\beta_1 \beta_2 < c^2 \lambda^{2n}$,
where $u := A_n A_{n-1} \cdots A_1 w$
and 
$$
\beta_1 := \frac{\|A_n \cdots A_1 v\|}{\|\wed^2 A_n \cdots A_1 \|^{1/2} \, \|v\|}
\quad \text{and} \quad 
\beta_2 := \frac{\|A_1^{-1} \cdots A_n^{-1} u\|}{\|\wed^2 A_1^{-1} \cdots A_n^{-1} \|^{1/2} \, \|u\|} \, .
$$
It follows that either $\beta_1 < c \lambda^n$ or $\beta_2 < c\lambda^n$ (or both).
If the first inequality holds then the desired conclusion follows directly from Lemma~\ref{le:ABD}.
If the second inequality holds, the lemma gives ${\theta_0} \in [-\delta_0,\delta_0]$ such that 
$$
R_{\theta_0} A_1^{-1} R_{\theta_0} A_2^{-1} R_{\theta_0} \cdots R_{\theta_0} A_n^{-1} R_{\theta_0} u
$$ 
is proportional to $v$,
and so replacing $\theta_0$ by $-\theta_0$ we obtain the desired conclusion.
\end{proof}

The fact that a rotation occurs both at the beginning and at the end of the matrix products in Lemma~\ref{le:ABD} and Corollary~\ref{co:ABD_bi} is somewhat inconvenient for our purposes. This is easily remedied by the following:
\begin{lemma}\label{le:quick_fix}
For all $\delta_1>0$ and $C>1$ there exists $\delta_0>0$ with the following property:
given matrices $A_1$, \ldots, $A_n \in \GL_2^+(\R)$ 
with $\|A_1\| \, \|A_1^{-1}\| \le C$, a vector $v \in \R^2$, and ${\theta_0} \in [-\delta_0,\delta_0]$,
there exists $\theta_1 \in [-\delta_1,\delta_1]$ such that the two vectors
$$
R_{\theta_0} A_n R_{\theta_0} A_{n-1} R_{\theta_0} \cdots R_{\theta_0} A_1 R_{\theta_0} v
\quad \text{and} \quad 
R_{\theta_1} A_n R_{\theta_1} A_{n-1} \cdots R_{\theta_1} A_1  v
$$ 
are proportional to one another.
\end{lemma}
\begin{proof}
Given $\delta_1>0$ and $C>1$, using compactness and continuity we may choose $\delta_0 \in (0, \delta_1/2]$ such that:
$$
\left.
\begin{array}{ll}
A_1 \in \GL_2^+(\R),         &\|A_1\| \cdot \|A_1^{-1}\| \le C     \\
u, v \in \R^2\setminus\{0\}, &\angle(u,v)     \le \delta_0
\end{array}
\right\}
\ \Rightarrow \ 
\angle(A_1 u, A_1 v) \le \delta_1/2 \, .
$$

Now let $A_1$, \dots, $A_n$, $v$, and ${\theta_0}$ be as in the statement of the lemma.
Assume that $v \neq 0$ and $\theta_0 \neq 0$, otherwise there is nothing to prove.
Assume also that $\theta_0>0$, the other case being analogous.
Note that
\begin{equation}\label{eq:note}
R_{\delta_0} A_1 R_{\delta_0} v \text{ belongs to the cone }
\big\{ t \, R_\theta A_1 v \colon t > 0, \ \theta \in [0,\delta_1] \big\} \, .
\end{equation}
Let $f$, $g \colon \R \ \to \R$ be continuous functions such that $f(0) = g(0)$ and
for all $\theta \in \R$,
\begin{align*}
R_{\theta} A_n R_{\theta} A_{n-1} R_\theta \cdots R_{\theta} A_1 R_{\theta} v &\text{ is proportional to } (\cos f(\theta), \sin f(\theta)) \, , \\[2mm]
R_{\theta} A_n R_{\theta} A_{n-1} \cdots R_{\theta} A_1 v &\text{ is proportional to } (\cos g(\theta), \sin g(\theta)) \, . 
\end{align*}
Then $f$ and $g$ are monotonically increasing.
It follows from \eqref{eq:note} that $f(\delta_0) \le g(\delta_1)$.
So, by the Intermediate Value Theorem, there exists $\theta_1 \in [0, \delta_1]$
such that $g(\theta_1) = f(\theta_0)$.
This proves the lemma.
\end{proof}

\begin{proof}[of Lemma~\ref{le:access_2}]
Combine Corollary~\ref{co:ABD_bi} with Lemma~\ref{le:quick_fix}.
\end{proof}

\begin{remark}\label{re:access_sufficient}
If 
$\sigma_1(A_n \cdots A_1) / \sigma_2(A_n \cdots A_1) < c^2 \lambda^{2n}$ 
then condition \eqref{eq:condition_v_w} from Lemma~\ref{le:access_2} 
holds for every pair of nonzero vectors $v$ and~$w$.
\end{remark}

\subsection{Higher-dimensional accessibility}

We now apply the previous results to derive a higher-dimensional result which is needed in the proof of Theorem~\ref{th:formula-GLd}.

\begin{lemma}\label{le:access_higher}
For all $\epsilon>0$ and $M>1$ there exist constants $c>0$ and $\lambda>1$ with the following property: 
if $1 \le p < d$ are integers and, 
$A_1$, \ldots, $A_m$ are matrices in $\GL_d(\R)$ with $\|A_i^{\pm 1}\| \le M$ for each $i$ 
and such that
\[
\frac{\sigma_{p}(A_m\cdots A_1)}{\sigma_{p+1}(A_m\cdots A_1)} < c^2 \lambda^{2m} \, ,
\]
then given any pair $E$, $F$ of subspaces of $\R^d$ with $\dim E = \codim F = p$,
there exist matrices $L_1$, \dots, $L_m \in \GL_d(\R)$ such that $\|L_i - A_i \|\leq  \epsilon$
and $(L_m \cdots L_1)(E) \cap F \neq \{0\}$.
\end{lemma}

%


\begin{proof} 
Given $\epsilon$ and $M$,
let $c$ and $\lambda$ denote the constants provided by Lemma~\ref{le:access_2} for the values $\delta_1 := \epsilon/M$ and $C := M^2$. 


Now fix integers $1 \le p < d$ and 
matrices $A_1$, \ldots, $A_m$ in $\GL_d(\R)$ such that $\|A_i^{\pm 1}\| \le M$ for each $i$ 
and $\sigma_p(P) / \sigma_{p+1}(P) < c^2 \lambda^{2m}$, where $P:=A_m\cdots A_1$.
Also fix subspaces $E$ and $F$ of $\R^d$ such that $\dim E = \codim F = p$.

Let $S\subseteq \R^d$ denote the span of the set of eigenvectors of $P^*P$ which correspond to eigenvalues less than or equal to $\sigma_p(P)^2$. We clearly have $\codim S < p$ and therefore $E \cap S \neq \{0\}$, so we may choose a unit vector $v \in E \cap S$.
Since $S$ admits an orthonormal basis consisting of eigenvectors of $P^*P$, by writing $v$ as a linear combination of these basis elements we may easily estimate
\[\|Pv\|^2 = \langle Pv,Pv\rangle = \langle P^*Pv,v\rangle \leq \sigma_p(P)^2\|v\|^2\]
so that $0<\|Pv\|\leq \sigma_p(P)$. Similarly let us define $U\subseteq \R^d$ to be the span of the set of eigenvectors of $P^*P$ which correspond to eigenvalues greater than or equal to $\sigma_{p+1}(P)^2$: a similar calculation  shows that $\|Pw\|\geq \sigma_{p+1}(P)\|w\|$ for every $w \in U$.
Since $P$ is invertible we have $\dim P(U) = \dim U >p$ and we may therefore choose a unit vector $w \in U$ such that $Pw \in F$. We thus have
\[
\frac{\|Pv\|}{\|Pw\|} < c^2\lambda^{2m}.
\]

Let $V_0 \subseteq \R^2$ denote the space spanned by $v$ and $w$, for each $i=1,\ldots,m$ define $V_i:=A_i \cdots A_1 V_0$, and let $\tilde{A}_i \colon V_{i-1}\to V_i$ denote the restriction of $A_i$ to $V_{i-1}$. We orient each plane $V_i$ so that the maps $\tilde{A}_i$ become orientation-preserving.
By Lemma~\ref{le:access_2}, there exists an angle $\theta_1 \in [-\delta_1, \delta_1]$
such that if $\tilde R_i$ denotes the rotation of the plane $V_i$ by angle $\theta_1$
then the nonzero vector 
$\tilde{R}_m A_m \tilde{R}_{m-1} \tilde{A}_{m-1} \cdots \tilde{R}_1 \tilde{A}_1 v$
is proportional to $Pw$.
Extend $\tilde{R}_i$ to the linear map $\hat{R}_i$ on $\R^d$ that equals the identity on the orthogonal complement of $\tilde{V}_i$. Then 
$\hat{R}_m A_m \hat{R}_{m-1} A_{m-1} \cdots \hat{R}_1 A_1v$
is proportional to $Pw$. Taking $L_i:=\hat{R}_i A_i$  we have $\|L_i - A_i\| \leq \frac{\epsilon}{M} \|A_i\| \leq \epsilon$ for every $i=1,\ldots,m$.
Since $v \in E$ and $Pw \in F$,
we achieve the desired conclusion. 
\end{proof}

\begin{remark}
It is also possible to prove Lemma~\ref{le:access_higher} by adapting some arguments from \cite{BoV}:
see \cite[\S~3.3]{B_minimal}.
\end{remark}


\section{Characterisation of discontinuities in dimension $2$}\label{se:GL2+}

The aim of this section is to prove formula~\eqref{eq:formula-GL2+},
which immediately implies Theorem~\ref{th:GL2+}.
That formula holds trivially when $\ell(A) = 1$,
so we are left to consider the case $\ell(A) = 2$.
We will actually prove the following slightly stronger fact:

\begin{lemma}\label{le:theorem}
If $\sA \in \cK(\GL_2^+(\R))$ is not $1$-dominated then
for every $\delta > 0$ there exists $\theta_0 \in [-\delta,\delta]$ such that
\begin{equation}\label{eq:the_theorem}
\subrad(R_{\theta_0} \sA) = \subrad(\wed^2\sA)^{1/2}.
\end{equation}
\end{lemma}

\begin{proof} 
Let $\sA \in \cK(\GL_2^+(\R))$.
Fix $B \in \sA$ such that $\det B =  \inf_{A \in \sA} (\det A)$,
which by \eqref{eq:min-det} equals $\subrad(\wed^2\sA)$.
Notice that one direction of inequality in \eqref{eq:the_theorem} is automatic, since 
$$
\subrad(R_\theta \sA) \ge \subrad(\wed^2(R_\theta \sA))^{1/2} = (\det B)^{1/2} 
\quad \text{for every~$\theta \in \R$.}
$$

Assume that $\sA$ is not $1$-dominated and let $\delta>0$ be arbitrary.
If there exists $\theta \in [-\delta,\delta]$ such that
the eigenvalues of $R_\theta B$ have equal absolute value then
$$
\subrad(R_\theta\sA) \leq \rho(R_\theta B) = (\det R_\theta B)^{1/2}=
(\det B)^{1/2} ,
$$
and there is nothing left to prove.
For the remainder of the proof we therefore assume that 
for each $\theta \in [-\delta, \delta]$,
the matrix $R_\theta B$ has eigenvalues $\mu_{1,\theta}$, $\mu_{2,\theta}$
with $|\mu_{1,\theta}| > |\mu_{2,\theta}|$.

For each $j=1$, $2$, let $v_{j,\theta}$ be an eigenvector of $R_\theta B$ corresponding to $\mu_{j,\theta}$,
chosen so that it has unit norm and depends smoothly on $\theta$. 
So there are smooth functions $f_j \colon [-\delta,\delta] \to \R$ such that
$$
v_{j, \theta} = (\cos f_j(\theta), \sin f_j(\theta)) \, .
$$
Clearly we may choose these functions so as to satisfy the additional inequality
$$
0 < f_2(0) - f_1 (0) < \pi .
$$

\begin{claim}
The functions $f_i$ are monotonic: $f_1$ is increasing and $f_2$ is decreasing.
\end{claim}

\begin{proof}[of the claim]
Let $h \colon \R \to \R$ be the unique smooth function such that
for all $\phi \in \R$, $\theta\in [-\delta,\delta]$, $j\in\{1,2\}$,
$$
\left\{
\begin{array}{l}
B (\cos \phi, \sin \phi) \text{ is proportional to } (\cos h(\phi), \sin h(\phi) )\, , \\[2mm]
\theta + h(f_j(\theta)) = f_j(\theta) \, . 
\end{array}
\right.
$$
Since $B$ has positive determinant it is clear that $h$ is an increasing function.
Moreover, for each $\theta \in [-\delta,\delta]$, the points 
$f_1(\theta)$ and $f_2(\theta)$ are fixed under the map $\phi \mapsto \theta + h(\phi)$,
the former being exponentially attracting and the latter being exponentially repelling.
So we have
$$
0 < h'(f_1(\theta)) < 1 \quad \text{and} \quad h'(f_2(\theta)) > 1 \, .
$$
Since
$f_j'(\theta) = [1 - h'(f_j(\theta))]^{-1}$ for both $j=1,2$,
it follows that $f_1'(\theta) > 0$ and $f_2'(\theta) < 0$.
\end{proof}

Let $c>0$ and $\lambda>1$ be the constants provided by Lemma~\ref{le:access_2}
in respect of $\delta_1 := \delta$
and $C:= \max_{A \in \sA} \|A\|\, \|A^{-1}\|$.
Since $\sA$ is not $1$-dominated, 
by Definition~\ref{de:dominated}(\ref{i:dom-singular}) there exists a finite sequence of matrices 
$A_1$, \dots, $A_n \in \sA$ such that:
$$
\frac{\sigma_1(A_n \cdots A_1)}{\sigma_2(A_n \cdots A_1)} < c^2 \lambda^{2n} \, .
$$
By Lemma~\ref{le:access_2} and Remark~\ref{re:access_sufficient}
we may choose $\theta_1 \in [-\delta,\delta]$ 
such that $Q_{\theta_1} v_{1,0}$ is proportional to $v_{2,0}$,
where 
$$
Q_\theta := R_\theta A_n \cdots R_\theta A_1.
$$

\begin{claim}
There exists $\theta_0 \in [-\delta,\delta]$ such that $Q_{\theta_0} v_{1,\theta_0}$
is proportional to $v_{2,\theta_0}$.
\end{claim}

\begin{proof}[of the claim]
If $\theta_1=0$ then there is nothing to prove, so we assume that $\theta_1\neq 0$. We shall give the proof when $\theta_1>0$, the argument in the case $\theta_1 < 0$ being analogous.

Let $g \colon [-\delta,\delta]^2 \to \R$ be the continuous function such that
$$
\left\{
\begin{array}{l}
Q_\theta v_{1,\phi} \text{ is proportional to } (\cos g(\theta,\phi), \sin g(\theta,\phi)) \, , \\[2mm]
f_2(0) - \pi < g(0,0) \le f_2(0).
\end{array}
\right.
$$
The function $g(\theta,\phi)$ is obviously increasing with respect to $\theta$,
and is also increasing with respect to $\phi$ (because $f_1(\phi)$ is). 

Since $Q_{\theta_1} v_{1,0}$ is proportional to $v_{2,0}$,
$$
g(\theta_1,0) = f_2(0) + m \pi \quad \text{for some } m \in \Z.
$$
Note that since $g$ is increasing in the first variable
$$
m \pi = g(\theta_1,0) - f_2(0) \ge g(0,0) - f_2(0) > -\pi \, ,
$$
and so $m \ge 0$.
Consider the function $h(\theta) := g(\theta,\theta) - f_2(\theta)$.
On the one hand $h(0) \le 0$, while on the other hand since $g$ is increasing in the second variable
$$
h(\theta_1) \ge g(\theta_1,0) - f_2(\theta_1) =  f_2(0) + m \pi - f_2(\theta_1) \ge m \pi \ge 0.
$$
Thus, by the Intermediate Value Theorem, 
there exists $\theta_0 \in [0,\theta_1]$ such that $h(\theta_0) = 0$.
\end{proof}

\begin{claim}
If $k$ is sufficiently large then 
$(R_{\theta_0} B)^k Q_{\theta_0}$ has non-real eigenvalues.
\end{claim}

\begin{proof}
The matrix of $(R_{\theta_0} B)^k Q_{\theta_0}$ 
with respect to the basis $(v_{1,\theta_0}, v_{2,\theta_0})$ is
$$
M_k = 
\begin{pmatrix}
\mu_{1,\theta_0}^k & 0 \\ 0 & \mu_{2,\theta_0}^k
\end{pmatrix}
\begin{pmatrix}
0 & b \\ c & d
\end{pmatrix} = 
\begin{pmatrix}
0 & b \mu_{1,\theta_0}^k \\ c \mu_{2,\theta_0}^k & d \mu_{2,\theta_0}^k
\end{pmatrix} \, ,
$$
where $b$, $c$, $d$ do not depend on $k$.
Thus, for sufficiently large $k$,
$$
\frac{(\tr M_k)^2}{4 \left|\det M_k\right|}  = 
\left| \frac{d^2}{4bc} \cdot \frac{\mu_{2,\theta_0}^k}{\mu_{1,\theta_0}^k}\right| < 1
$$
and so $M_k$ has non-real eigenvalues.
\end{proof}
	
In particular $\rho((R_{\theta_0} B)^k Q_{\theta_0}) 
=  \big( \det (R_{\theta_0} B)^k Q_{\theta_0} \big)^{1/2}$
for sufficiently large $k$.
It follows that:
\begin{equation*}
\log \subrad(R_{\theta_0} \sA) \le \liminf_{k \to \infty} \frac{1}{k+n} \log \rho((R_{\theta_0} B)^k Q_{\theta_0})  
=   \liminf_{k \to \infty} \frac{k \log \det B + \log \det Q_{\theta_0} }{2(k+n)} 
=   \frac{\log \det B}{2} \, .
\end{equation*}
This proves the lemma.
\end{proof}

\begin{remark}\label{re:monotonicity}
The alert reader will have noticed that monotonicity plays an essential role in the proof above.
This property, which is also important in the proof of Lemma~\ref{le:ABD} in \cite{ABD},
breaks down if we allow negative determinants, 
and ceases to make sense in higher dimensions.
We will discuss related matters in \S\ref{ss:fixed_card} below.
\end{remark}

\section{Characterisation of discontinuities in arbitrary dimension}\label{se:iff-GLd}

In this section we prove Theorem~\ref{th:formula-GLd}:
if $\ell(\sA)$ is the smallest index of domination for $\sA \in \cK(\GL_d(\R))$
then
$\liminf_{\sB \to \sA} \subrad(\sB) =
\subrad\left(\wed^{\ell(\sA)} \sA\right)^{1/\ell(\sA)}$.

\subsection{Outline}

Let $\sA \in \cK(\GL_d(\R))$ be given, and for convenience let us write $\ell:=\ell(\sA)$. It is sufficient to prove that $\liminf_{\sB \to \sA} \subrad(\sB) \le \subrad\left(\wed^{\ell} \sA\right)^{1/\ell}$, since the inequality in the other direction is a direct consequence of Theorem~\ref{th:lipschitz}, as explained in the introduction. We can also assume that $\subrad(\sA)>\subrad(\wed^{\ell}\sA)^{1/\ell}$, otherwise there is nothing to prove.
Recall that the quantity $\|\wed^\ell A\|^{1/\ell}$ is precisely the geometric mean of the first $\ell$ singular values of the matrix $A$. The inequality between the two lower spectral radii therefore asserts, in effect, that if a product $P$ of elements of $\sA$ approximately minimises $\|\wed^\ell P\|$ then the first singular value of $P$ must be significantly larger than the geometric mean of the first $\ell$ singular values of $P$. We will show that by slightly enlarging the set $\sA$, we may find a nearby product $P'$ such that the geometric mean of the first $\ell$ singular values is similar to that of $P$, but such that those singular values are more closely distributed around their geometric mean. By iterating this construction we bring these singular values so closely into alignment with one another that the first singular value must closely approximate the geometric mean of the first $\ell$ singular values, and the lower spectral radius of the perturbed version of $\sA$ may in this manner be reduced arbitrarily close to $\subrad(\wed^\ell \sA)^{1/\ell}$ by an arbitrarily small perturbation.

\smallskip

The technical steps involved are roughly as follows. Recall that in Example~\ref{ex:simple} and in the proof of Theorem~\ref{th:GL2+} we constructed products with small norm by taking a long product $P$ and composing it with a short product $R$ which transported the more expanding eigenspace of $P$ onto the more contracting eigenspace of $P$ in such a manner that the 
absolute value of the eigenvalues of $PR$ both coincided with $\sqrt{|\det PR|}$.
At the core of Theorem~\ref{th:formula-GLd} is a higher-dimensional version of this principle which is summarised in Lemma~\ref{le:e-f-bound} below: given a matrix $P \in \GL_d(\R)$ and an integer $p$ in the range $1 \leq p<d$, there exist subspaces $E$, $F$ of $\R^d$ such that if $R \in \GL_d(\R)$ satisfies $R(E) \cap F \neq \{0\}$ then the norm $\|\wed^p PRP\|$ is bounded above by the reduced quantity $(\sigma_{p+1}(P)/\sigma_p(P))\|\wed^p P\|^2$ up to a multiplicative factor depending only on $R$. 
(This idea originates in \cite{BoV}.)
We will apply this result in combination with Lemma~\ref{le:access_higher} above, which shows that if $\sA$ is not $p$-dominated, then given such a product $P$ of elements of $\sA$ the desired matrix $R$ can be constructed as a product of matrices close to $\sA$ with an a priori bound on $R$ depending on the desired degree of closeness. 

The reader will notice that this procedure is only directly useful for reducing the norms of products of elements of $\sA$ if the product $P$ may be chosen in such a way that $\sigma_2(P)/\sigma_1(P)\ll1$, which may fail to be possible when the dimension $d$ exceeds two: for example, if $d \geq \ell =4$ then it could be the case that the first two singular values of $P$ are equal to one another and exceed the geometric mean of the first four singular values, whilst the third and fourth singular values are much smaller than the geometric mean. In such an instance the above argument does not directly allow us to find a product $PRP$ whose first singular value is smaller than that of $P$ relative to the length of the product. Instead the appropriate procedure is to apply the above argument with $p=2$, creating a nearby product belonging to a perturbed set $\sA'$ whose second singular value is greatly reduced: by applying the argument a second time to this new perturbed set with $p=1$, we finally succeed in reducing the first singular value and hence the lower spectral radius. A key feature of this procedure is the observation that the choice $p=2$ marks a large disagreement between successive singular values which can be productively exploited to bring the singular values closer to their geometric mean: the existence of such a `pivot' $p$ is given as Lemma~\ref{le:pivot} below.

\smallskip

The sketch above suggests an algorithmic way to construct discontinuities. 
Our actual proof is more direct than this but is less constructive: rather than repeatedly perturbing $\sA$ by appending to it finite sets of nearby matrices, we simply expand $\sA$ to include \emph{all} matrices within distance $\varepsilon$ of $\sA$ and show that this has the same effect as performing the above sequence of perturbations arbitrarily many times.
To facilitate this technical shortcut, we consider some especially convenient quantities 
\eqref{eq:zeta} and \eqref{eq:Z} that are tailored to measure disagreements between singular values
and lower spectral radii.

\subsection{Setup}

In order to formalise these arguments we require some notation. For each $A \in \GL_d(\R)$ and $k=1,\ldots,d$ we denote the logarithm of the $k^{\mathrm{th}}$ singular value by
\[\lambda_k(A):=\log \sigma_k(A),\]
and the total of the logarithms of the first $k$ singular values by
\[\tau_k(A):=\sum_{i=1}^k \lambda_i(A)=\log(\sigma_1(A)\cdots \sigma_k(A)) = \log\left\|\wed^kA\right\|.\]
We define also $\tau_0(A):=0$. To measure the amount of agreement between the first $k$ singular values of a matrix $A$  we use the following device which was introduced by the first named author in \cite{B_minimal}. Given $A \in \GL_d(\R)$ and $k \in \{2,\ldots,d\}$ let us define
\begin{equation}\label{eq:zeta}
\zeta_k(A):=\tau_1(A)+\tau_2(A) + \cdots +\tau_{k-1}(A) - \left(\frac{k-1}{2}\right)\tau_k(A).
\end{equation}
It may be found helpful to interpret the function $\zeta_k$ visually as follows. Given an integer $k \in \{2,\ldots,d\}$, consider the graph of the function $[0,k] \to \R$ defined by $i \mapsto \tau_i(A)$ for the integers $i=0,\ldots,k$ and by affine interpolation on each of the intervals $[i,i+1]$. Since the sequence $\lambda_i(A)$ is nonincreasing, this graph is concave, and in particular it lies above (or on) the line from $(0,0)$ to $(k,\tau_k(A))$. The quantity $\zeta_k(A)$ is precisely the area of the region between the graph associated to $A$ and the line from $(0,0)$ to $(k,\tau_k(A))$. See Figure~\ref{fig:zeta}.
Geometrically it is clear that that $\zeta_k(A) \geq 0$, and that $\zeta_k(A)=0$ if and only if the first $k$ singular values of $A$ are equal, i.e., if and only if $\|A\|=\|\wed^k A\|^{1/k}$. 
More precisely, we have
\begin{equation}\label{eq:ineq1}
\zeta_k(A) \geq \frac{1}{2} \big( k \tau_1(A) - \tau_k(A) \big)= \frac{1}{2} \Big( k \log\|A\| - \log\left\|\wed^k A \right\| \Big) \, ; 
\end{equation}
geometrically this means that the area of the triangle 
with vertices $(0,0)$, $(1,\tau_1(A))$ and $(k, \tau_k(A))$ is at most $\zeta_k(A)$;
see Figure~\ref{fig:ineq1}.

\begin{figure}[htb]
\centering
\begin{minipage}{0.4\textwidth}
\centering
\begin{tikzpicture}[scale=1]
\draw [<->] (0,3)--(0,0)--(4.5,0);
\draw[-,thick,pattern=horizontal lines gray] (0,0)--(1,2)--(2,2.5)--(3,2.25)--(4,1.5);
\draw[dashed] (0,0)--(4,1.5);
\end{tikzpicture}
\caption{Shaded area equals \eqref{eq:zeta}.}\label{fig:zeta}
\end{minipage}
\qquad  
\begin{minipage}{0.4\textwidth}
\centering
\begin{tikzpicture}[scale=1]
\draw [<->] (0,3)--(0,0)--(4.5,0);
\draw[-,thick] (0,0)--(1,2)--(2,2.5)--(3,2.25)--(4,1.5);
\draw[dashed] (0,0)--(4,1.5);
\draw[dashed,pattern=horizontal lines gray] (0,0)--(1,2)--(4,1.5);
\end{tikzpicture}
\caption{Shaded area equals RHS of \eqref{eq:ineq1}.}\label{fig:ineq1}
\end{minipage}
\end{figure}

To measure the alignment of singular values on products of elements of $\sA$ whose $\ell(\sA)^{\mathrm{th}}$ exterior power is small we use the following device.
For each $\delta>0$ define
\begin{equation}\label{eq:Z}
Z_\delta(\sA):=\liminf_{n \to \infty} \inf_{\substack{A_1,\ldots,A_n \in \sA\\  \|\wed^{\ell(\sA)} A_n\cdots A_1\|\leq e^{n\delta}\subrad(\wed^{\ell(\sA)}\sA)^n }} \frac{1}{n}\zeta_{\ell(\sA)}(A_n\cdots A_1).
\end{equation}
The following coarse estimate is sufficient to allow us to pass from upper bounds on $Z_\delta(\sA)$ to the approximate agreement of $\subrad(\sA)$ with $\subrad(\wed^{\ell(\sA)}\sA)^{1/\ell(\sA)}$.
\begin{lemma}\label{le:z-delta}
Let $\sA \in \cK(\GL_d(\R))$ and $\delta>0$. Then
\[\log\subrad(\sA) \leq \frac{1}{\ell(\sA)}\log\subrad\left(\wed^{\ell(\sA)}\sA\right) +Z_\delta(\sA)+\delta.\]
\end{lemma}
\begin{proof}
Define $\ell:=\ell(\sA)$. The case $\ell=1$ being trivial, we assume $\ell\geq 2$. Choose $A_1,\ldots,A_n \in \sA$ such that
\[\frac{1}{n}\log\left\|\wed^{\ell}A_n\cdots A_1\right\| \leq \log\subrad\left(\wed^{\ell}\sA\right)+\delta\]
and
\[\frac{1}{n}\zeta_{\ell}\left(A_n\cdots A_1\right) \leq Z_\delta(\sA)+\frac{\delta}{2}.\]
Using inequality \eqref{eq:ineq1}, we estimate
\begin{align*}
\log\subrad(\sA)&\leq \frac{1}{n}\log\left\|A_n\cdots A_1\right\|\\
&\leq \frac{1}{n\ell}\log\left\|\wed^\ell A_n\cdots A_1\right\| + \frac{2}{n\ell}\zeta_\ell\left(A_n\cdots A_1\right)\\
&\leq \frac{1}{\ell} \log\subrad\left(\wed^{\ell}\sA\right)+\frac{\delta}{\ell}+\frac{2}{\ell}\left(Z_\delta(\sA)+\frac{\delta}{2}\right)\\
&\leq \frac{1}{\ell} \log\subrad\left(\wed^{\ell}\sA\right)+Z_\delta(\sA)+\delta \, ,
\end{align*}
as claimed.
\end{proof}


We will need the following two lemmas.
The first one originates as \cite[Lemma 2.6]{B_minimal},
while the second one is a special case of \cite[Lemma 3.7]{B_minimal}.

\begin{lemma}\label{le:pivot}
For each $d \geq 2$ there exists a constant $\alpha_d \in (0,1)$ with the following property: if $P \in \GL_d(\R)$ and $2 \leq \ell \leq d$ then there exists an integer $p$ such that $1 \leq p <\ell$ and
\begin{equation}\label{eq:pivot}
	\frac{\lambda_p(P)-\lambda_{p+1}(P)}{2} \geq \alpha_d \zeta_\ell(P).
\end{equation}
\end{lemma}


\begin{lemma}\label{le:e-f-bound}
Let $P \in \GL_d(\R)$ and $1 \leq p <d$. Then there exist subspaces $E$ and $F$ of $\R^d$ such that $\dim E= \codim F=p$ with the following property: if $R \in \GL_d(\R)$ satisfies $R(E)\cap F \neq \{0\}$, then
\[\tau_p(PRP) \leq 2\tau_p(P)-\lambda_p(P)+\lambda_{p+1}(P)+C_d\left(1+\log \|R\|\right)\]
where $C_d>1$ is a constant that depends only on $d$.
\end{lemma}

\subsection{The proof}

Given $\sA \in \cK(\GL_d(\R))$ and $\epsilon>0$ we will find it convenient to write
\[\sA_{\epsilon}:=\left\{B \in M_d(\R) \colon \exists A \in \sA \text{ such that }\|B-A\|\leq \epsilon\right\}.\]
We note that if $\epsilon>0$ is sufficiently small then $\sA_\epsilon \in \cK(\GL_d(\R))$.
and moreover $\ell(\sA_\epsilon)=\ell$. Indeed, it follows from Definition~\ref{de:dominated}(\ref{i:dom-multicone}) that the property of being $\ell$-dominated is open in $\cK(\GL_d(\R))$, and so if $d_H(\sA,\sB)$ is sufficiently small then $\sB$ is $\ell$-dominated and therefore $\ell(\sB) \leq \ell$. On the other hand it is also clear from Definition~\ref{de:dominated}(\ref{i:dom-singular}) that the relation $\sA \subset \sA_\epsilon$ implies that $\sA_\epsilon$ cannot be $i$-dominated when $\sA$ is not, and therefore $\ell(\sA_\epsilon)=\ell$ when $\epsilon$ is sufficiently small. 
We shall always assume $\epsilon>0$ to be small enough that this is the case.

\smallskip

The two previous results combine with Lemma~\ref{le:access_higher} to yield the following estimate on $Z_\delta$ which forms the core of the proof of the theorem:
\begin{lemma}\label{le:synthesis}
Let $\sA \in \cK(\GL_d(\R))$ be such that $\ell(\sA)>1$.
Let $\delta_2>\delta_1>0$. 
Then for all sufficiently small $\epsilon>0$ we have $\ell(\sA_\epsilon) = \ell(\sA)$ and
\[Z_{\delta_2}(\sA_\epsilon) \leq(1- \alpha_d) Z_{\delta_1}(\sA)+\frac{\ell(\sA)\delta_1}{2}\]
where $\alpha_d \in (0,1)$ is the constant from Lemma~\ref{le:pivot}.
\end{lemma}

\begin{proof}
Fix $\delta>0$ and $\sA \in \cK(\GL_d(\R))$ and define $\ell:=\ell(\sA)$. 
Choose constants $\kappa_1,\kappa_2>0$ such that
\begin{equation}\label{eq:kappa-1}\delta_1+2\kappa_1<\delta_2\end{equation}
and
\begin{equation}\label{eq:kappa-2}(1-\alpha_d)\kappa_2+\frac{3(\ell-1)\kappa_1}{2}<\frac{1}{2}\delta_1.\end{equation}

Since $\wed^\ell\sA$ is $1$-dominated, it follows from Theorem~\ref{th:lipschitz} that if $\epsilon>0$ is sufficiently small then we have
\begin{equation}\label{eq:subrad-close}
\log \subrad\left(\wed^{\ell}\sA_\epsilon\right) \geq 
\log \subrad\left(\wed^{\ell}\sA\right) - \kappa_1.
\end{equation}
For the remainder of the proof we fix $\epsilon>0$ small enough that the above properties hold. 

To demonstrate the claimed bound on $Z_{\delta_2}(\sA_\epsilon)$ we must show that there exist infinitely many integers $n \geq 1$ for which there exists a product $\tilde{A}_n \cdots \tilde{A}_1$ of $n$ matrices in $\sA_\epsilon$ such that 
\begin{equation}\label{eq:wedge-ell-property}
\frac{1}{n}\log\left\|\wed^\ell \tilde{A}_n \cdots \tilde{A}_1\right\| < \log\subrad\left(\wed^\ell\sA_\epsilon\right)+\delta_2
\end{equation}
and
\begin{equation}\label{eq:area-property}
\frac{1}{n}\zeta_{\ell}\left(\tilde{A}_n\cdots \tilde{A}_1\right)<(1-\alpha_d) Z_{\delta_1}\left(\sA\right)+\frac{\ell\delta_1}{2}.
\end{equation}

For each $p \in \{1, \dots, \ell-1\}$, since $\sA$ is not $p$-dominated, 
there exist an integer $m_p>0$ and matrices $B_1^{(m)}$, \dots, $B_{m_p}^{(m)}$ 
such that
\begin{equation}\label{eq:boring}
\frac{\sigma_p \left( B_{m_p}^{(m)} \cdots B_1^{(m)}\right)}{\sigma_{p+1} \left( B_{m_p}^{(m)} \cdots B_1^{(m)}\right)} < c^2 \lambda^{2 m_p} 
\end{equation}
where $c>0$ and $\lambda>1$ are the constants provided by Lemma~\ref{le:access_higher} in respect of 
the chosen value of $\epsilon$ and $M:=\sup\{\|A^{\pm 1}\| \colon A \in \sA_\epsilon\}$.
Let $\overline{m}:= \max\{m_1, m_2, \dots, m_{\ell-1}\}$.
Choose an integer $r$ and a finite sequence of matrices $A_1$, \dots, $A_r\in\sA$ such that
\begin{equation}\label{eq:r-big-1}\frac{1}{r}\log\left\|\wed^\ell A_r\cdots A_1\right\| < \log \subrad\left(\wed^\ell \sA\right)+\delta_1,\end{equation}
\begin{equation}\label{eq:r-big-2}\frac{1}{r}\zeta_\ell(A_r\cdots A_1)<Z_{\delta_1}(\sA)+\kappa_2,\end{equation}
and
\begin{equation}\label{eq:r-big-3}\frac{C_d\ell(1+\overline{m}\log M)}{2r}<\kappa_1\end{equation}
where $C_d>1$ is the constant provided by Lemma~\ref{le:e-f-bound}, noting that $r$ may if required be taken to be arbitrarily large. 
Let $P:=A_r \cdots A_1$. By Lemma~\ref{le:pivot} there exists $p \in \{1,\ldots,\ell-1\}$ such that
\begin{equation}\label{eq:area-reduction}\zeta_\ell(P)-\frac{\lambda_p(P)-\lambda_{p+1}(P)}{2} \leq (1-\alpha_d)\zeta_\ell(P).\end{equation}
Let $E$ and $F$ be the subspaces of $\R^d$ provided by Lemma~\ref{le:e-f-bound} for the matrix $P$ and integer $p$. 
In view of \eqref{eq:boring}, Lemma~\ref{le:access_higher} provides matrices
$L_1,\ldots,L_m \in \sA_\epsilon$ such that $(L_m\cdots L_1)(E)\cap F \neq \{0\}$,
where $m := m_p$.
Define $R := L_m \cdots L_1$. 

We claim that $PRP$ is the desired product $\tilde{A}_n \cdots \tilde{A}_1$, where $n:=2r+m$. To establish \eqref{eq:wedge-ell-property} we may directly estimate
\begin{align*}\frac{1}{2r+m}\log \left\|\wed^\ell PRP\right\| &<\frac{1}{r}\log\left\|\wed^\ell P\right\| + \frac{\ell m\log M}{2r+m}\\
&\leq \log\subrad(\wed^\ell \sA) + \delta_1+\kappa_1\\
&\leq \log\subrad(\wed^\ell \sA_\epsilon) + \delta_1+2\kappa_1\\
&\leq \log\subrad(\wed^\ell \sA_\epsilon) + \delta_2
\end{align*}
using respectively the elementary bound $\|R\|\leq M^m$, \eqref{eq:r-big-1}, \eqref{eq:r-big-3}, \eqref{eq:subrad-close} and \eqref{eq:kappa-1}.
To establish \eqref{eq:area-property} we proceed by estimating the values $\tau_k(PRP)$ individually. By Lemma~\ref{le:e-f-bound} we have
\[\tau_p(PRP) \leq 2\tau_p(P)-\lambda_p(P)+\lambda_{p+1}(P)+C_d(1+\log \|R\|)\]
and hence
\[\frac{1}{2r+m}\tau_p(PRP) \leq \frac{1}{r}\left(\tau_p(P) - \frac{\lambda_p(P)-\lambda_{p+1}(P)}{2}\right) + \kappa_1\]
using \eqref{eq:r-big-3}.  
For integers $k$ such that $1 \leq k \leq \ell-1$ and $k \neq p$ we directly estimate
$$
\frac{1}{2r+m}\tau_k(PRP)= \frac{1}{2r+m}\log\left\|\wed^k PRP\right\|
\leq \frac{2}{2r+m} \tau_k(P) + \frac{km \log M}{2r+m}
 \leq \frac{1}{r}\tau_k(P)+\kappa_1$$
using \eqref{eq:r-big-3} again.
In the case of $\tau_\ell(PRP)$ we instead estimate from below
$$
\frac{1}{2r+m}\tau_\ell(PRP) \geq \log\subrad\left(\wed^\ell\sA_\epsilon\right) \geq \log\subrad\left(\wed^\ell\sA \right) - \kappa_1 
\geq \frac{1}{r}\tau_\ell(P) - \delta_1 - \kappa_1
$$
using respectively \eqref{eq:subrad-close} and \eqref{eq:r-big-1}.
Combining all of our estimates on the numbers $\tau_k(PRP)$ yields
\begin{align*}\frac{1}{2r+m}\zeta_\ell (PRP) &=\frac{1}{2r+m}\left( \sum_{k=1}^{\ell-1} \tau_k(PRP) - \left(\frac{\ell-1}{2}\right)\tau_\ell(PRP)\right)\\
&\leq \frac{1}{r}\left(\sum_{k=1}^{\ell-1}\tau_k(P) - \left(\frac{\ell-1}{2}\right)\tau_\ell(P)\right)- \frac{\lambda_p(P)-\lambda_{p+1}(P)}{2r} + \frac{\left(\ell-1\right)\left(3\kappa_1+\delta_1\right)}{2}\\
&\leq \frac{(1-\alpha_d)}{r}\zeta_\ell(P) +\frac{\left(\ell-1\right)\left(3\kappa_1+\delta_1\right)}{2}\\
&\leq (1-\alpha_d)Z_{\delta_1}(\sA) +\frac{\ell\delta_1}{2}\end{align*}
using respectively \eqref{eq:area-reduction}, \eqref{eq:r-big-2} and \eqref{eq:kappa-2}, by which means we have arrived at \eqref{eq:area-property}. Since $r$ may be taken arbitrarily large for fixed $\epsilon$, we have $Z_{\delta_2}(\sA_\epsilon) \leq Z_{\delta_1}(\sA)+\ell(\sA)\delta_1/2$ as required and the proof of Lemma~\ref{le:synthesis} is complete.
\end{proof}

\begin{proof}[of Theorem~\ref{th:formula-GLd}]
As explained in the beginning of this section, it suffices to consider the case
where $\sA \in \cK(\GL_d(\R))$ satisfies
\[
\subrad(\sA)>\subrad\left(\wed^{\ell(\sA)}\sA\right)^{1/\ell(\sA)} \, .
\]
Define $\ell:=\ell(\sA)$, which is necessarily greater than $1$. We have seen in Lemma~\ref{le:synthesis} that there exists $\varepsilon_0$ such that if $\varepsilon \in (0,\varepsilon_0]$ then $\ell(\sA_\varepsilon)=\ell$. Define
\[c := \sup_{(\epsilon, \delta) \in (0,\epsilon_0] \times (0, \infty)}  Z_\delta(\sA_\epsilon)\]
which is clearly finite. We claim that in fact $c=0$.

To prove this assertion let us suppose instead that $c>0$. The quantity $Z_\delta(\sA_\epsilon)$
is monotone non-increasing with respect to each of the variables $\epsilon \in (0,\epsilon_0]$
and $\delta \in (0, \infty)$, so in particular
\[\lim_{(\epsilon, \delta) \to (0,0)} Z_\delta(\sA_\epsilon) = c.\]
Take $(\bar\epsilon, \bar\delta) \in (0,\epsilon_0] \times (0, \infty)$ such that
$$
\left.
\begin{array}{r} 
\epsilon \in  \left(0,\bar{\epsilon}\right] \\[2mm]
\delta \in \left(0,\bar{\delta}\right]
\end{array}
\right\}
\ \Rightarrow \ 
\left(1 - \frac{\alpha_d}{2} \right) c <  Z_\delta(\sA_\epsilon) \le c .
$$
Now let $\delta_1 \in (0, \bar\delta)$ be small enough that $\ell \delta_1 < \alpha_d c$. 
Applying Lemma~\ref{le:synthesis} to $\sA_{\bar\epsilon/2}$ we may choose $\epsilon \in (0,\frac{\bar\epsilon}{2}]$ small enough that
$$
Z_{\bar\delta} (\sA_{\bar\epsilon/2 + \epsilon}) \le (1-\alpha_d) Z_{\delta_1}(\sA_{\bar\epsilon/2}) + \frac{\ell \delta_1}{2} \, .
$$
We then have
\[\left(1-\frac{\alpha_d}{2}\right)c <Z_{\bar\delta}\left(A_{\bar\epsilon/2+\epsilon}\right)\le (1-\alpha_d) Z_{\delta_1}(\sA_{\bar\epsilon/2}) + \frac{\ell \delta_1}{2} \leq (1-\alpha_d)c + \frac{\alpha_d c}{2}\]
which is a contradiction, and we conclude that $c=0$ as claimed.

Now, by \eqref{eq:wedges} we trivially have $\subrad(\wedge^\ell \sA_\epsilon)^{1/\ell}\leq \subrad( \sA_\epsilon) $ for every $\epsilon>0$. On the other hand since $Z_\delta(\sA_\epsilon)=0$ whenever $\delta>0$ and $\varepsilon \in (0,\epsilon_0]$ it follows from Lemma~\ref{le:z-delta} that conversely $\subrad( \sA_\epsilon)\leq \subrad(\wedge^\ell \sA_\epsilon)^{1/\ell}$ for every $\varepsilon \in (0,\epsilon_0]$, and so when $\epsilon$ belongs to this range the two quantities must be equal. Since each of the sets $\wedge^\ell \sA_\varepsilon$ is $1$-dominated we conclude using Theorem~\ref{th:lipschitz} that
\[
\liminf_{\sB \to \sA} \subrad(\sB) \le 
\lim_{\varepsilon \to 0} \subrad(\sA_\varepsilon) =\lim_{\epsilon \to 0}\subrad\left(\wedge^\ell \sA_\varepsilon\right)^{1/\ell} = \subrad\left(\wedge^\ell \sA\right)^{1/\ell} \, .
\]
This completes the proof of Theorem~\ref{th:formula-GLd}.
\end{proof}

\section{Further examples of discontinuity}\label{se:examples}

Corollary~\ref{co:iff-GLd} characterises completely the points $\sA$ of discontinuity 
of the lower spectral radius as those sets $\sA$ which satisfy
\begin{equation*} 
	\subrad(\sA) > \subrad \left( \wed^{\ell(\sA)} \sA \right)^{1/\ell(\sA)} \, .
\end{equation*}
It may be however difficult to verify this condition in concrete situations.
In this article we have so far presented only one example where this condition is satisfied, namely Example~\ref{ex:simple}. It is therefore instructive to look for more examples.

\medskip
We first observe that we may easily extend Example~\ref{ex:simple} to higher dimensions:

\begin{example}\label{ex:2}
Let $\sB_1 \in \cK(\GL_{d_1}(\R))$ and $\sB_2 \in \cK(\GL_{d_2}(\R))$, and suppose that there exists $\lambda \in \R$ such that
\[ \inf_{B_1 \in \sB_1} \sigma_{d_1}(B_1) >\lambda> \left(\inf_{\substack{B_1 \in \sB_1\\B_2 \in \sB_2}}  |(\det B_1)(\det B_2)|\right)^{\frac{1}{d_1+d_2}}.\]
Choose arbitrary matrices $R_1 \in O(d_1)$ and $R_2 \in O(d_2)$,
and define $\sA \in \cK(\GL_{d_1+d_2}(\R))$ by
\[\sA = \left\{\left(\begin{array}{cc} B_1&0\\0&B_2\end{array}\right)\colon B_1 \in \sB_1\text{ and }B_2 \in \sB_2\right\} \cup \left\{\left(\begin{array}{cc}\lambda R_1&0\\0&\lambda R_2\end{array}\right)\right\}.\]
We claim that $\ell(\sA)=d_1+d_2$ and $\subrad(\sA)>\subrad(\wed^{d_1+d_2}\sA)^{\frac{1}{d_1+d_2}}$ so that $\sA$ is a discontinuity point of the lower spectral radius.  

Let us justify these claims.
Since $\sA$ contains a scalar multiple of an isometry it is obviously not $k$-dominated for any $k<d_1+d_2$, so we have $\ell(\sA)=d_1+d_2$ as claimed. It is clear that any product of $n$ elements of $\sA$ forms a block diagonal matrix whose upper-left entry has norm at least $\lambda^n$, and it follows easily that $\subrad(\sA)=\lambda$. On the other hand we have
$$
\subrad\left(\wed^{d_1+d_2}\sA\right)^{\frac{1}{d_1+d_2}}=\left(\inf_{A \in \sA}|\det A|\right)^{\frac{1}{d_1+d_2}}
=\left(\inf_{\substack{B_1 \in \sB_1\\B_2 \in \sB_2}} 
|(\det B_1)(\det B_2)|\right)^{\frac{1}{d_1+d_2}}<\lambda=\subrad(\sA)
$$
as stated.
\end{example}

Note that if $\sB_1 \in \cK(\GL_{d_1}(\R))$ and $\sB_2 \in \cK(\GL_{d_2}(\R))$ are arbitrary then the sets $t\sB_1$ and $\sB_2$ will always meet the above criteria when $t>0$ is sufficiently large. In particular we may construct uncountably many sets of matrices which are discontinuity points of $\subrad$ and which are not pairwise similar. 
Nevertheless these examples are quite particular in the sense that there is a common invariant splitting.

In order to give more interesting examples, let us return to dimension~$2$.
Let us show how Example~\ref{ex:simple} may be generalised so as to replace the identity transformation on $\R^2$ with an arbitrary rational rotation.

\begin{example}\label{ex:still_simple}
Let $R = R_{q\pi/p}$, where $p>0$ and $q$ are relatively prime integers.
We will explain how to find $H \in \GL_2(\R)$ such that the set $\sA = \{R,H\}$ (which is obviously not $1$-dominated) satisfies $\subrad(\sA) > \subrad(\wedge^2\sA)^{1/2}$.

Choose arbitrarily a positive $\delta<\pi/(2p)$ and a one-dimensional subspace $V$ of $\R^2$.
Let $\cD$ be the set of non-zero vectors in $\R^2$ whose angle with $V$ is less than or equal to $\delta$.
Define $\cC := \cD \cup R \cD \cup \cdots \cup R^{p-1}\cD$ so that we have $R \cC = \cC$.
Let $W$ be a one-dimensional subspace of $\R^2$ that does not intersect $\cC$,
for example, $R_{\pi/(2p)} V$.
Let $P \in M_d(\R)$ be the projection with image $V$ and kernel $W$. Clearly there exists $\kappa>0$ such that $\|Pv\|\geq 2\kappa\|v\|$ for all $v \in \cC$.
If $\tilde{P}$ is a invertible matrix sufficiently close to $P$ then $\tilde{P}\cC \subseteq \cD$, $|\det \tilde{P}| < \kappa^2$ and $\|\tilde{P}v\|\geq \kappa\|v\|$ for all $v \in\cC$. Fix one such matrix $\tilde P$ and define $H:=\kappa^{-1}\tilde{P}$ and $\sA:=\{R,H\}$. We observe that $|\det H| <1$ so in particular $\subrad(\wedge^2\sA)=\min\{|\det H|, \det R\}<1$. On the other hand if $v \in \cC$ then $\|Rv\|=\|v\|$, $\|Hv\|\geq \|v\|$ and both $Rv$ and $Hv$ belong to $\cC$, so it is easily seen that every product of the matrices $R$ and $H$ has norm at least $1$. In particular we have $\subrad(\sA)=1>\subrad(\wedge^2\sA)^{1/2}$ as claimed.
\end{example}

It is interesting to note that the set $\cC$ in the example above 
satisfies all requirements from Definition~\ref{de:dominated}(\ref{i:dom-multicone})
for being a $1$-multicone, except for \emph{strict} invariance.
Such `weak $1$-multicones' cannot exist
if $\sA = \{ H, R_\theta\}$ with $\theta/\pi$ irrational.
In fact, it is an open problem whether $\subrad(\sA)>\subrad(\wed^2\sA)^{1/2}$
is possible or not in the irrational case.
We develop this question further in \S~\ref{ss:fayad}.

\section{Open questions and directions for future research}\label{se:future}

\subsection{Sets with non-invertible matrices}

The results proved in this article are valid for compact subsets of $\GL_d(\R)$. It is not clear how to extend those results to subsets of $M_d(\R)$. Let us indicate one of the difficulties, related with the notion of domination.
In circumstances where we have demonstrated discontinuity of the lower spectral radius we have done so using Definition~\ref{de:dominated}(\ref{i:dom-singular}), whereas when we have demonstrated continuity of the lower spectral radius we have used Definition~\ref{de:dominated}(\ref{i:dom-multicone}). For sets of non-invertible matrices these two properties can fail to be equivalent, as the following example illustrates:
\begin{example}\label{ex:non-dom}
Define $\sA \in \cK(M_d(\R))$ by
\[\sA:=\left\{\left(\begin{array}{ccc}2&0&0\\0&0&0\\0&0&1\end{array}\right),\left(\begin{array}{ccc}0&0&0\\0&2&0\\0&0&1\end{array}\right)\right\}.\]
Then $\sA$ is $1$-dominated in the sense of Definition~\ref{de:dominated}(\ref{i:dom-singular}) but does not satisfy Definition~\ref{de:dominated}(\ref{i:dom-multicone}). Furthermore it is not an interior point of the set of all $1$-dominated matrix sets in the sense of Definition~\ref{de:dominated}(\ref{i:dom-singular}), since for example pairs of the form
\[\sB:=\left\{\left(\begin{array}{ccc}2&0&0\\0&2^{-m}&0\\0&0&1\end{array}\right),\left(\begin{array}{ccc}0&0&0\\0&2&0\\0&0&1\end{array}\right)\right\}\]
are clearly not $1$-dominated in this sense.
\end{example}

\subsection{Continuity on sets of fixed cardinality}\label{ss:fixed_card}

In the proof of Theorem~\ref{th:formula-GLd} we were able to show that if \eqref{eq:iff-continuity} is not satisfied for a fixed set $\sA \in \cK(\GL_d(\R))$ then there exist perturbed sets $\sB \in \cK(\GL_d(\R))$ arbitrarily close to $\sA$ such that the lower spectral radius of $\sB$ is less than that of $\sA$ by a constant amount. If the original set $\sA$ has finite cardinality, however, this theorem does not yield any information about the cardinality of the perturbed set $\sB$. Theorem~\ref{th:no-min} illustrates the interest of being able to show that the perturbed set $\sB$ can be chosen with the same cardinality as $\sA$. It is therefore natural to ask whether Theorem~\ref{th:GL2+} -- in which the perturbed set has equal cardinality to the unperturbed set -- extends to $\cK(\GL_2(\R))$, and whether it admits an analogue for higher-dimensional matrices. 

We have already observed in Remark~\ref{re:monotonicity} that the assumption of positive determinant
is essential to our proof of Theorem~\ref{th:GL2+}.
The following example demonstrates that this theorem indeed does not extend directly to subsets of $\GL_2(\R)$:
\begin{example}\label{ex:no-discontinuity}
Define
\[\sA:= \{ A_1, A_2 \}, \quad \text{where} \quad
A_1 := \left(\begin{array}{cc}2&0\\0&\frac{1}{8}\end{array}\right),
\quad
A_2:= \left(\begin{array}{cc}1&0\\0&-1\end{array}\right).
\]
Then the map $\theta \mapsto R_\theta \sA$ is continuous at $\theta=0$.

Let us justify this assertion. Let
\[A_{1,\theta}:=R_\theta A_1,\qquad A_{2,\theta}:=R_\theta A_2,\]
and define $\cC:=\{(x,y) \in \R^2 \setminus \{0\} \colon |y| \le |x|\}$. If $|\theta|$ is sufficiently small then the cone $A_{1,\theta}\cC$ is strictly contained in $\cC$ and  $A_{2,\theta}A_{1,\theta}\cC\subseteq \cC$. The matrix $A_{1,\theta}$ increases the Euclidean norm of every element of $\cC$ and the matrix $A_{2,\theta}$ is an isometry: since furthermore $A_{2,\theta}^2=\mathrm{Id}$ it follows that a product $B_n\cdots B_1$ of elements of $R_\theta\sA$ such that $B_1=A_{1,\theta}$ cannot decrease the length of an element of $\cC$ and hence must have norm at least $1$. Since $\rho(A_{2,\theta})=1$ we deduce easily that $\subrad(R_\theta\sA)=1$ when $|\theta|$ is sufficiently small.
\end{example}
The set $\sA$ given in Example~\ref{ex:no-discontinuity} nonetheless is a discontinuity point of $\subrad$ on the set of all pairs of $\GL_2(\R)$-matrices. To see this, consider the set $\sA_n := \{A_1, H_n\}$ where
$$
P_n := \begin{pmatrix} 1 & e^{-2n} \\ e^{-2n} & 1 \end{pmatrix}, \quad
H_n := P_n^{-1} \begin{pmatrix} e^{-1/n} & 0 \\ 0 & -e^{1/n} \end{pmatrix} P_n.
$$
Direct calculation shows that $H_n^{2n^2}$ maps the horizontal axis onto the vertical axis.
Thus in a manner similar to Example~\ref{ex:simple} we see that $\subrad(\sA_n) = \frac{1}{2}$. 
This example emboldens us to make the following conjecture:

\begin{conjecture}\label{conj:continuity}
For each $n$, $d \geq 1$ the function $\subrad \colon \GL_d(\R)^n \to \R_+$ is continuous at $\sA$ if and only if
\[\subrad(\sA)=\subrad\left(\wed^{\ell(\sA)}\sA\right)^{1/\ell(\sA)}\]
where $\ell(\sA)$ is the smallest index of domination for $\sA$. 
\end{conjecture}
If this conjecture is valid then the set
\[\left\{\sA \in \GL_d(\R)^n \colon\subrad(\sA)=\subrad\left(\wed^{\ell(\sA)}\sA\right)^{1/\ell(\sA)}\right\}\]
is precisely the set of continuity points of the upper semi-continuous function $\subrad \colon \GL_d(\R)^n \to \R_+$, and hence is a dense $G_\delta$ set. Clearly to compute $\subrad(\sA)$ on this set it is sufficient to compute the lower spectral radius of $\wed^{\ell(\sA)}\sA$, and so in the generic case the computation of $\subrad$ would -- subject to the validity of Conjecture~\ref{conj:continuity} -- be reduced to the computation of the lower spectral radii of $1$-dominated sets.

\subsection{The Lebesgue measure of the discontinuity set}\label{ss:fayad}

In the previous discussions we have noted that since the lower spectral radius is upper semi-continuous, its points of continuity form a residual set, and are thus `large' in a topological sense. In an alternative direction, for finite sets $\sA \subset \GL_d(\R)$ of fixed cardinality $n$ we could ask how large is the set of continuity points of $\subrad$ in the sense of Lebesgue measure on $\GL_d(\R)^n$. In the case of pairs of $2 \times 2$ matrices with positive determinant we believe that the set of continuity points is much smaller in the sense of Lebesgue measure than it is in the topological sense.

\begin{definition}\label{de:resist}
Let $\mathcal{H}\subset \SL_2(\R)$ denote the set of all matrices with distinct real eigenvalues, and $\mathcal{E}\subset \SL_2(\R)$ the set of all matrices with distinct non-real eigenvalues. We say that the pair $(H,R) \in \mathcal{H}\times \mathcal{E}$ \emph{resists impurities} if there exist constants $\epsilon,\lambda>0$ depending on $(H,R)$ such that if $A_n\cdots A_1$ is a product of the matrices $H$ and $R$ which features at most $\epsilon n$ instances of $R$, then $\|A_n\cdots A_1\|\geq e^{\lambda n}$. 
\end{definition}
The following conjecture was introduced in \cite[\S5.3]{BF}; some partial results may be found in \cite{AR,FK}.
\begin{conjecture}\label{conj:bf}
The set 
of pairs $(H,R) \in \mathcal{H}\times\mathcal{E}$ that resist impurities
has full Lebesgue measure in $\mathcal{H}\times\mathcal{E}$.
\end{conjecture}
The above conjecture, if correct, implies that the set of discontinuities of $\subrad$ on $\GL_2(\R)^2$ is large in the sense of Lebesgue measure. We note:
\begin{proposition}
Conjecture~\ref{conj:bf} is equivalent to the following statement: the set of discontinuities of $\subrad \colon \GL_2(\R)^2\to \R_+$ on the open set
\[\mathcal{U}:=\left\{(\alpha H,\beta R)\in \GL_2^+(\R)^2 \colon H\in\mathcal{H},\text{ }R\in \mathcal{E}\text{ and }\beta>\alpha>0\right\}\]
has full Lebesgue measure in that set.
\end{proposition}

\begin{proof} 
For all $\sA \in \GL_2(\R)^2$, we have
$\subrad(\sA) \ge \subrad(\wedge^2\sA)^{1/2}$.
Since the left-hand side of this inequality is an upper-semicontinuous function,
and the right-hand side is continuous by \eqref{eq:min-det},
the points $\sA$ where equality holds are points of continuity of $\subrad \colon \GL_2(\R)^2\to \R_+$.
The converse to this statement is true and follows by Theorem~\ref{th:GL2+}.
We therefore must show that Conjecture~\ref{conj:bf} is valid if and only if almost every $\sA=(\alpha H, \beta R) \in \cU$ satisfies $\subrad(\sA)>\alpha$. 

Let us first show that Conjecture~\ref{conj:bf} implies the claimed statement. Clearly the set of all $(\alpha H, \beta R) \in \cU$ such that $(H,R)$ resists impurities has full Lebesgue measure, so it suffices to show that every such pair has lower spectral radius greater than $\alpha$. Given such $H,R,\alpha,\beta$ let $\epsilon,\lambda$ be the constants associated to the pair $(H,R)$. Let $A_n\cdots A_1$ be a product of elements of the matrices $\alpha H$ and $\beta R$, which we assume to contain exactly $k_1$ instances of $\alpha H$ and exactly $k_2$ instances of $\beta R$. Write $A_n\cdots A_1=\alpha^{k_1}\beta^{k_2}B_n\cdots B_1$ where each $B_i$ is either $H$ or $R$. If $k_2 \leq \epsilon n$ then
\[\|A_n\cdots A_1\| \geq \alpha^n \|B_n\cdots B_1\|\geq e^{\lambda n}\alpha^n,\]
and if $k_2>\epsilon n$ then
\[\|A_n \cdots A_1\| \geq \alpha^{k_1}\beta^{k_2} \geq \alpha^{(1-\epsilon)n} \beta^{\epsilon n},\]
so
\[\subrad(\sA) \geq \min\left\{e^\lambda \alpha,\alpha^{1-\epsilon}\beta^{\epsilon}\right\}>\alpha\]
as required.

Conversely let us suppose that the set
\[\{(\alpha H, \beta R) \in \cU \colon \subrad(\{\alpha H, \beta R\})>\alpha\}\]
has full Lebesgue measure in $\cU$. By integrating over the parameters $\alpha,\beta$ and using Fubini's theorem we deduce that there exist $\alpha_0,\beta_0$ such that $\beta_0>\alpha_0>0$ and the set
\[\left\{(H,R) \in \mathcal{H} \times \mathcal{E} \colon \subrad\left(\{\alpha_0 H, \beta_0 R\}\right)>\alpha_0\right\}\]
has full Lebesgue measure in $\mathcal{H}\times\mathcal{E}$. It suffices to show that every element of this set resists impurities. Given such $(H,R)$ let us write $\subrad(\left\{\alpha_0 H,\beta_0 R\right\})=e^{2\delta} \alpha_0>\alpha_0$. If $(\alpha_0 H,\beta_0 R)$ does not resist impurities then for each $\varepsilon>0$ we may find a product $B_n\cdots B_1$ of the matrices $H$ and $R$, with $k_1$ instances of the former matrix and $k_2<\varepsilon n$ instances of the latter matrix, such that $\|B_n\cdots B_1\| \leq e^{n\delta}$. Hence
\[e^{2n\delta}\alpha_0^n =\subrad\left(\{\alpha_0 H, \beta_0 R\}\right)^n \leq \alpha_0^{k_1}\beta_0^{k_2}\|B_n \cdots B_1\| \leq  \alpha_0^{(1-\varepsilon)n}\beta_0^{\varepsilon n} e^{n\delta}\]
and therefore
\[ e^\delta\alpha_0<e^{2\delta}\alpha_0 \leq e^\delta \alpha_0^{1-\varepsilon}\beta_0^\varepsilon.\]
Letting $\varepsilon \to 0$ yields a contradiction, and so that $(H,R)$ resists impurities as claimed.
\end{proof}
Motivated by this correspondence we pose the following weaker version of Conjecture~\ref{conj:bf}:
\begin{conjecture}
The set of discontinuities of $\subrad \colon \GL_2(\R)^2 \to \R_+$ has positive Lebesgue measure.
\end{conjecture}
It is known that pairs of matrices $(H,R_\theta)$ which resist impurities exist, and indeed by the above arguments it may be seen that the pairs of matrices in Example~\ref{ex:still_simple} have this property when their determinants are normalised to $1$. Unfortunately these are essentially the only examples known to us. These examples have the property that $\theta / \pi$ is rational in such a manner that powers of $R_\theta$ cannot map the expanding subspace of $H$ onto, or even arbitrarily close to, the contracting subspace of $H$. If Conjecture~\ref{conj:bf} is true then for typical $(H,R) \in \mathcal{H} \times \mathcal{E}$,
although there are products that map the expanding subspace of $H$ 
(or that of any other matrix in $\mathcal{H}$ formed from products of matrices $R$ and $H$)
very close to its contracting subspace, these products must contain many instances of $R$.

\begin{ack}
We thank Christian Bonatti and Micha{\l} Rams for important discussions. We also thank the two referees for a number of corrections and suggestions.
\end{ack}



\affiliationone{Jairo Bochi \\
Departamento de Matem\'{a}tica, Pontif\'{i}cia Universidade Cat\'{o}lica do Rio de Janeiro, Rua Mq.\ S.\ Vicente 225, Rio de Janeiro\\
Brazil
}
\affiliationtwo{Ian D. Morris\\
Department of Mathematics, University of Surrey, Guildford GU2 7XH\\
United Kingdom\\
\email{i.morris@surrey.ac.uk}
}
\affiliationthree{%
Current address:\\
Facultad de Matem\'{a}ticas, Pontificia Universidad Cat\'{o}lica de
Chile, Av.\ Vicu\~{n}a Mackenna~4860, Santiago\\
Chile
}
\affiliationfour{~} 
\end{document}